\documentclass{alggeom}
\usepackage{ams math}

\usepackage{psfrag}
\usepackage{setspace}
\usepackage{paralist}

\usepackage{multicol}
\usepackage{color}
\usepackage{array,amscd,amsmath,amsthm,ifthen}
\usepackage{amssymb}
\usepackage{graphicx}
\usepackage[all]{xy}
\usepackage{multicol}
\usepackage{cite}

\def\E{\mathcal{E}}

\def\I{\mathcal{I}}

\def\cO{\mathcal{O}}

\def\R{\mathcal{R}}

\def\bP{\mathbb{P}}
\def\bC{\mathbb{C}}
\def\bZ{\mathbb{Z}}

\def\lra{\longrightarrow}

\def\ov#1{\overline{#1}}

\def\coker{\mathrm{Coker}}

\def\cliff{\mathrm{Cliff}}
\def\Pic{\operatorname{Pic}}

\def\Ext{\operatorname{Ext}}
\def\ext{\mathcal E xt}
\def\spec{\operatorname{Spec}}

\def\rk{\operatorname{rk}}
\def\Sing{\operatorname{Sing}}

\def\Im{\operatorname{Im}}
\def\cliff{\operatorname{Cliff}}
\def\wt{\widetilde}

\newtheorem{thm}{Theorem} 
\newtheorem{prop}[thm]{Proposition}
\newtheorem{lemma}[thm]{Lemma}
\newtheorem{cor}[thm]{Corollary}

\newtheorem{rem}[thm]{Remark}

\newtheorem{defin}[thm]{Definition}

\newcommand{\be}{\begin{equation}}
\newcommand{\ee}{\end{equation}}

\begin{document}
\title {On hyperplane sections of K3 surfaces} 
\author{E. Arbarello}
\email{ea@mat.uniroma1.it}
\address{Dipartimento di Matematica Guido Castelnuovo, Universit\`a di Roma Sapienza, 
Piazzale A. Moro 2, 00185 Roma, Italia. }
\author{A. Bruno}
\email{bruno@mat.uniroma3.it}
\address{Dipartimento di Matematica e Fisica, Universit\`a Roma Tre, L.go S.L. Murialdo 1, 00146 Roma}
\author{E. Sernesi}
\email{sernesi@mat.uniroma3.it}
\address{Dipartimento di Matematica e Fisica, Universit\`a Roma Tre, L.go S.L. Murialdo 1, 00146 Roma}

\thanks{The first named author would like to thank the Institute of Advanced Study for hospitality during 
the preparation of part of this work. The second and third  named authors were partly supported by 
 the project MIUR-PRIN 2010/11  ''Geometria delle Variet\`a Algebriche''.}

\begin{abstract} Let $C$ be a Brill-Noether-Petri curve   of genus $g\geq 12$. We prove that  $C$ lies on a polarized K3 surface, or on a limit thereof, if and only if the Gauss-Wahl map for $C$ is not surjective. The proof  is obtained by studying the validity of two conjectures by J. Wahl.
Let $\I_C$ be the ideal sheaf of a non-hyperelliptic, genus $g$, canonical curve. The first conjecture
states that, if  $g\geq 8$,  and if the Clifford index of $C$ is greater than 2, then $H^1(\bP^{g-1}, \I_C^2(k))=0$, for $k\geq 3$. We prove this conjecture for $g\geq 11$.
The second conjecture states that a Brill-Noether-Petri curve of genus $g\geq 12$ is extendable if and only if $C$ lies on a K3 surface. As observed in the Introduction, the correct version of this conjecture should admit limits of polarised K3 surfaces in its statement.
This is what we prove in the present work.

\end{abstract}

\maketitle
\vspace*{6pt}

\section{Introduction}

The fundamental result of Brill-Noether theory for curves, as proved by Griffiths and Harris in 
\cite{GH80}, states that, on a general curve of genus $g$, there is no $g^r_d$ as soon as the Brill-Noether number

$$
\rho(g,r,d)=g-(r+1)(g-d+r)
$$
is negative. As usual by a $g^r_d$ we mean a  linear series of degree $d$ and dimension $r$ associated to an $(r+1)$-dimensional  subspace $V$ of the space of global sections of a line bundle $L$ of degree $d$ on $C$.   Associated to $L$ is the so-called Petri map:
$$
\mu_{0,L}: H^0(C, L)\otimes H^0(KL^{-1})\longrightarrow H^0(C, K)
$$
where $K=K_C$ is the canonical bundle on $C$. The Brill-Noether number is nothing but the difference between the dimensions of the codomain and the domain of this map. A stronger version of the Brill-Noether theorem, as proved by Gieseker in \cite{Gie82},  states that, on a general curve of genus $g$, the Petri map $\mu_{0,L}$ is injective for every line bundle $L$ on $C$. A curve for which the Petri map $\mu_{0,L}$ is injective for every line bundle $L$ on it, is said to be  a Brill-Noether-Petri curve or, in short, a BNP-curve (cf. \cite{ACGH84}, \cite{ACG10}).
A breakthrough in  Brill-Noether Theory came when Lazarsfeld \cite{Laz86} proved Gieseker's theorem by showing that 
a genus $g$ curve $C$ lying on a K3 surface $S$ having the property that $\Pic(S)=\bZ\cdot C$ is a BNP curve. Lazarsfeld's theorem also shows that it is not possible to characterise the locus of curves lying on K3 surfaces by looking at line bundles on them.
One may look at vector bundles and then the situation is completely different. In fact, following an idea of Mukai, the authors of this paper where able to describe a general K3 surface in terms of an exceptional Brill-Noether locus 
for vector bundles on  one of its hyperplane sections \cite{ABS14}.
 Here we take a different point of view.
 \vskip 0.5 cm
 The starting point is a  theorem by  Wahl \cite{Wah87} proving that, for a  smooth curve $C$ of genus $g\geq2$
 lying on a K3 surface $S$, the Gaussian map (also called the Gauss-Wahl map or Wahl map)
 $$
 \aligned
 \nu: \wedge^2H^0&(C, K)\longrightarrow H^0(C, 3K)\\
 &s\wedge t\mapsto s\cdot dt-t\cdot ds
 \endaligned
 $$
 is not surjective. A beautiful geometrical proof of this fact was given by Beauville and Merindol \cite{BM87}. This theorem is all the more significant in light of the fact that for a general curve the Wahl map is surjective. This was proved, with different methods, by Ciliberto, Harris, Miranda \cite{CHM88}, and by Voisin \cite{Vois92}. In this last paper Voisin, for the first time,  suggests to study the non-surjectivity of the Wahl map, under the Brill-Noether-Petri condition, in order to characterize hyperplane sections of K3 surfaces (\cite{Vois92}, Remark 4.13 (b), where the author also refers to Mukai). In \cite{CU93}
  Cukierman and Ulmer proved that, when $g=10$, the closure of the locus in $M_{10}$ of curves  that are  
 canonical section of  K3 surfaces, coincides with the closure of the locus of curves with non-surjective Wahl map. This locus has been studied in more detail by Farkas--Popa in \cite{FP05} where they gave a counterexample to the slope conjecture.
  \vskip 0.5 cm
  Our aim is to prove the following theorem.
  
  \begin{thm}\label{Main1} Let $C$ be a Brill-Noether-Petri curve  of genus $g\geq 12$. Then $C$ lies on a polarized K3 surface, or on a limit thereof, if and only if its Wahl map is not surjective.
  \end{thm}
 
 Our investigation was sparked by a remarkable paper by  
 Wahl \cite{Wah97} where the author analyses the significance of the non-surjectivity of his map from the point of view of  deformation theory.

Let $X=C\subset\bP^{g-1}$ 
be  a canonically embedded curve of genus $g\geq3$. Wahl finds a precise connection between the cokernel of $\nu$
and the deformation theory of the affine cone over $C$. Denoting by
$$
A=\oplus\Gamma(C, \cO_C(k))
$$
the coordinate ring of this cone, Wahl considers the graded cotangent module $T^1_A$ and  shows that
$$
(T^1_A)_{-1}\cong (\coker\,\nu)^\vee
$$
He also gives a precise interpretation of the graded pieces of the obstruction module $T^2_A$:
$$
(T^2_A)_k\cong H^1(\bP^{g-1}, \I_{C/\bP}^2(1-k))^\vee\,.
$$
where $\I_C$ is the ideal sheaf of $C\subset\bP^{g-1}$. 
The modules $T^1_A$ and $T^2_A$ govern the deformation theory
of the cone $\spec A$.  The canonical curve $C\subset\bP^{g-1}$ is said to be {\it extendable} if it is a hyperplane section of a surface  $\ov S\subset \bP^{g}$ which is not a cone. The main theorem of \cite{Wah97} is the following.

\begin{thm} {\rm(Wahl).}\label{wahl1} Let $C\subset\bP^{g-1}$ be a canonical curve of genus $g\geq 8$, with $\cliff(C)\geq 3$.
Suppose that 
\be\label{vanishing}
H^1(\bP^{g-1}, \I_{C/\bP}^2(k))=0\,,\quad k\geq 3
\ee
Then $C$ is extendable if and only if $\nu$ is not surjective.

\end{thm}
Wahl then conjectures that every canonical curve, of Clifford index greater than or equal to 3, satisfies the vanishing condition (\ref{vanishing}). We   recall that the Clifford index of a curve $C$ is the minimum
value of $\deg D-2\dim |D|$ taken over all linear systems $|D|$ on $C$ such that $\dim|D|\geq1$ and 
$\dim|K_C-D|\geq1$.
With a slight restriction on the genus, this is what we prove in the first part of the present paper.

\begin{thm}\label{T:main}
Let $C\subset \bP^{g-1}$ be a canonically embedded curve of $g \ge 11$.  Suppose that  $\cliff(C)\ge 3$. Let $\I_{C/\bP}\subset \cO_{\bP^{g-1}}$ be the ideal sheaf  of $C$.  Then:
\[
H^1(\bP^{g-1},\,\I_{C/\bP}^2(k))=0
\]
for all $k \ge 3$.
\end{thm}

From Theorem \ref{wahl1}  one  gets the obvious  corollary:

\begin{cor} Let $C\subset\bP^{g-1}$ be a canonical curve of genus $g\geq 11$, with $\cliff(C)\geq 3$.
Then $C$ is extendable if and only if $\nu$ is not surjective.
\end{cor}

To say that $C$ is extendable means that there exists a projective surface $\ov S\subset \bP^g$, not a cone, having $C$
as a hyperplane section. In general the surface $\ov S$ has isolated singularities and it is only when there is a smoothing
of $\ov S$ in $\bP^g$, that we can say that $\ov S$ is  {\it the limit of a  K3 surface}.

In his paper Wahl gives a nice example of an extendable curve $C\subset \ov S\subset \bP^g$ for which $\ov S$ is non-smoothable. In the example, the curve $C$ has a realisation as a smooth plane  curve of degree $\ge 7$ and therefore is highly non-BNP.
This is perhaps one of the reasons that lead Wahl to conjecture that : {\it``A Brill-Noether-Petri curve of genus $g\geq 8$
sits on a K3 surface if and only if  the Gaussian  is not surjective''}. This conjecture, as stated, is slightly incorrect. Indeed, 
in the two recent papers \cite{ABFS16} and \cite{AB16} it is shown that there are plane curves (the du Val curves)
that are BNP  curves, with non-surjective Gauss-Wahl map (for every  value of the genus), which can not be realised as  hyperplane sections of a smooth K3 surface (for every odd value of the genus greater than 12). 
Thus, from this point of view, the statement of Theorem \ref {Main1} is the correct one.

\vskip 0.3 cm

We now proceed to give  a brief presentation of Part 1 and Part 2 of our paper.
\vskip 0.8 cm

\centerline{Part 1}
\vskip 0.5 cm 
One of the  results in \cite{Wah97} is that, indeed, the vanishing condition (\ref{vanishing}) holds  for a general curve
of genus $g\geq 8$. The way Wahl proves this vanishing theorem is by showing that this is true   for a class of  pentagonal curves 
of genus $g\geq 8$. Recall that associated to a complete pencil of degree $d$ on $C$, in  short, a complete $g^1_d$, is a scroll $X(g^1_d)\subset \bP^{g-1}$ containing $C$.  The set-theoretical description of $X(g^1_d)$ is given by:
$$
X(g^1_d)=\underset{D\in g^1_d}\cup \langle D\rangle
$$
where $\langle D\rangle\cong\bP^{d-2}$ is the linear span in $\bP^{g-1}$ of the divisor $D$.
The scroll $X(g^1_d)$ is a $(d-1)$- dimensional  variety of degree $g-d+1$ whose desingularization
is a projective bundle over $\bP^1$. For $g\geq 8$, it is possible to choose a curve $C$ with a base-point-free $g^1_5$ such that the scroll $Y=X(g^1_5)$ is  a smooth $4$-fold. Wahl proves that to check the vanishing statement (\ref{vanishing}) for such a curve $C$, it is enough to verify 
the analogous statement 
\be\label{vanishing-y}
H^1(Y, \I_{C/Y}^2(k))=0\,,\quad k\geq 3
\ee
where $\I_{C/Y}$ is the defining ideal sheaf of $C$ on $Y$. For the $4$-fold $Y$ it is not difficult to compute  a resolution of  the ideal sheaf $\I_{C/Y}$ and the vanishing result (\ref{vanishing-y}) follows easily.

\vskip 0.3 cm
Our proof of Theorem \ref{T:main} evolves, broadly,  along the same lines  of Wahl's proof of the vanishing (\ref{vanishing})  in the pentagonal case. The questions are: what is the right 
scroll that could serve the same purpose as $X(g^1_5)$?  and how can one use the hypothesis that
$\cliff(C)\geq 3$? The answer to both questions comes from a remarkable theorem by  Voisin, which generalizes  to the case $\cliff(C)\geq 3$ a theorem by  Green and Lazarsfeld regarding curves with $\cliff(C)\geq 2$. This is Theorem \ref{T:voisin} below, which characterises curves with $\cliff(C)\geq 3$. This characterisation consists of
 the existence on $C$ of a base-point-free $g^1_{g-2}$ whose residual is a  base-point-free $g^2_{g}$
giving a birational realization of $C$ in $\bP^2$ as an irreducible  curve of degree $g$ with only double points
as  singularities. It is then natural  to look at the scroll $X=X(g^1_{g-2})\subset \bP^{g-1}$. A slight disadvantage of this scroll is that it is singular; a more serious disadvantage  is that it is of high dimension (equal to
$g-3$). On the positive side, $X$ is of low degree, in fact it is  a cubic projecting the Segre variety $\bP^1\times\bP^2\subset\bP^5$, and $C$ lies entirely in the non-singular locus of $X$.
As in the pentagonal case, the reduction of the vanishing statement  (\ref{vanishing})  to the vanishing statement
\be\label{vanishing-x}
H^1(X, \I_{C/X}^2(k))=0\,,\quad k\geq 3
\ee
is rather uneventful (see Section \ref{T:main}).
The heart of the matter is to prove (\ref{vanishing-x}). 
\vskip 0.3 cm
What comes to our aid is the plane representation of $C$ via the $g^2_g$ of Voisin's theorem (see Section \ref{surface-P}).
Let $\Gamma\subset\bP^2$ be the plane curve given by this $g^2_g$. Via the adjoint linear system to 
$\Gamma$, the projective plane is mapped birationally to a surface $P\subset\bP^{g-1}$ with isolated singularities.
This surface is   contained in $X$ and contains $C$ in its smooth locus. One then looks at the diagram
(\ref{diag-P-X}) linking the conormal bundle to $C$ in $X$ to the restriction to $C$ of the  conormal bundle to $P$ in $X$  (in diagram (\ref{diag-P-X}) we treat the case $k=3$ but the case $k\geq 3$ is analogous).
The vanishing of $H^1(X, \I^2_{C/X}(3))$ is equivalent to the surjectivity of the natural  homomorphism 
$H^0(X,  \I_{C/X}(3))\to H^0(C, N^\vee_{C/X}(3H))$. Looking at the exact sequence
$$
0\to N^\vee_{P/X|C}(3H) \to  N^\vee_{C/X}(3H)\to \cO_C(3H-C)\to 0 
$$
it is not hard to show that this, in turn, is equivalent to the surjectivity of the map $H^0(X,  \I_{P/X}(3))\to H^0(C, N^\vee_{P/X|C}(3H))$.
What comes next is a rather unexpected fact, namely  that a suitable twist of the restricted conormal bundle $N^\vee_{P/X|C}$ is trivial (see Section \ref{conormal}). This makes all the computations straightforward and brings the proof to a quick conclusion (see Section \ref{square1}).
\vskip 0.5 cm
The cohomology of the square of an ideal sheaf has attracted quite a bit of interest
and the reader may consult \cite{BEL91}, \cite{Ber97}, \cite{Ver02} for related results.

\vskip 2 cm

\centerline{Part 2}
\vskip 0.5 cm

In the second part we start from an {\it extendable canonical curve $C$}.
This means that there exists a projective surface $\ov S\subset \bP^g$, not a cone, having $C$
as a hyperplane section. In general the surface $\ov S$ has isolated singularities. These surfaces have been completely classified by Epema \cite{E83}, \cite{E84} (his work, the works by
Ciliberto and Lopez \cite{CL02}, du Val \cite{DV33} and Umezu \cite{Um81}, were all important in our study). The classification of these surfaces runs as follows. 
Consider the diagram
\be
\xymatrix{
S\ar[d]_p\ar[r]^q&\ov S\subset\bP^g\\
S_0
}
\ee
where $q$ is the minimal resolution and $S_0$ a minimal model. 
Epema proves    (see Theorem \ref{epema} below)
that  $\ov S$ is projectively normal, that $h^0(S, -K_S)=1$, and that
 $S$ is either a  $K3$ surface or a ruled surface.  The smooth surface  $S$ is equipped with a  polarisation $C$, with $C^2=2g-2$, $ h^0(S, C)=g+1$
and with an anti-canonical divisor $Z\sim-K_S$, with  $C\cdot Z=0$.  The linear system $|C|$ maps $S$ to $\ov S$ and, when $Z\neq0$, it contracts $Z$ to one or two isolated singularities, depending on whether $Z$ possesses one or two connected components.
 If $ S_0$ is not a smooth K3 surface, then either $S_0\cong \bP^2$ or $S_0\cong\bP(E)$, where $E$ is a rank two vector bundle over a curve $\Gamma$.
 
 \vskip 0.3 cm
 
The task at hand becomes clear. In order to prove Theorem \ref{Main1} we must show that if the hyperplane section $C$ of $\ov S$ is a Brill-Noether-Petri curve, then $\ov S$ is a (possibly singular) K3 surface or it is smoothable in $\bP^g$ as a surface with canonical sections. 
 \vskip 0.3 cm

  \vskip 0.3 cm
A  necessary condition for this to happen is that the  singularities of $\ov S$ are smoothable.  A first  step consists in showing that, in fact, this is all we need to check. 
Indeed a general deformation theory argument,
based on the nature of surfaces with canonical sections, shows that once the local smoothability holds, then the surface  $\ov S$ is smoothable in $\bP^g$ as a surface with canonical sections.
\vskip 0.3 cm 
Let us see how the BNP condition reflects itself in local smoothability of the  singularities of $\ov S$. An example helps to elucidate the situation. For this we limit ourselves to the case in which $S_0\cong\bP^2$. It easily follows from the above description of $\ov S$, $S$ and $S_0$ that the plane model of $C$ is a curve $C_0$ of degree $d$ whose singularities all lie  on cubic $J_0$ (which is nothing but the image in $\bP^2$ of the anti-canonical curve $Z$). Let us assume for simplicity that $C_0$ has only ordinary multiple points.
 The surface $S$ is obtained from $S_0$ by blowing up all the points of intersection $C_0\cap J_0$. If $h$ is the cardinality of  $C_0\cap J_0$, then $Z^2=9-h$. It is well known that when $Z^2$ is very negative, then the singularity of $\ov S$, obtained by contracting $Z$, is not smoothable. Suppose now that $C_0$ has only double points as singularities  and let $\delta$  be their number. Certainly, when $g$ is big,  $h=3d-\delta$ can be  much bigger than $9$, so that the number $Z^2$ can be very negative, and $\ov S$  not smoothable. How is this prevented from a  Brill-Noether condition? Simply by B\'ezout's theorem, telling us that $3d\geq 2\delta$, so that the genus formula gives $2g\geq d^2-6d+2$,
 whereas the Brill-Noether requirement for a $g^2_d$ is that $3d\geq2g+6$, leaving us with a meagre  $d\leq 8$.
   \vskip 0.3 cm
  By general arguments one also sees that the elliptic singularity of $\ov S$
   is smoothable as soon as $Z^2\geq-9$.

  \vskip 0.3 cm
  What all of this teaches us  is that the BNP-condition has a definite effect on the number of singular points of $C_0$ and therefore on the number $Z^2$. In fact in our analysis for the case $S_0\cong\bP^2$, we see that, as soon as $g\geq 12$, once we force the BNP condition on $C$ we are left, in each genus,  with 
 only five possible cases for the  plane curves $C_0$,
  (one of which, for example, is a plane degree-$3g$ curve with $8$ points of multiplicity $g$ and one point of  multiplicity $g-1$) and that in all of these cases we get
  $Z^2\geq-6$.
        \vskip 0.3 cm
  The Brill-Noether-Petri tool to obtain such a drastic reduction of cases is rather rudimentary. We use time and again the pencil of cubic curves passing through eight of the points of maximal multiplicity of $C_0$, 
  and we exclude cases  by insisting that this pencil should cut out on $C$ a BNP linear series.
   
    \vskip 0.3 cm
    
    Finally two words about the cases where  $S_0\cong\bP(E)$ with $E$ a rank-2 bundle over a smooth curve $\Gamma$. The case when the genus of $\Gamma$ is bigger than 1 is solved swiftly via an elementary Brill-Noether argument. The genus-zero case is  reduced to the case of $\bP^2$. The genus-one case proceeds exactly as the case of $\bP^2$: only, it is  much easier.

\vskip 0.8 cm

\part{}

\vskip 0.8 cm

\section{The scroll}

We assume all schemes to be defined over $\bC$.  We let $C$ be a projective nonsingular irreducible curve of genus $g \ge 11$.  Recall the following definition.

\begin{defin}\label{D:petri}   Let  $L$ be an invertible sheaf  on $C$.  The natural map:
\[
\mu_0(L): H^0(C,L)\otimes H^0(C,\omega_CL^{-1}) \lra H^0(C,\omega_C)
\]
 is called the \emph{Petri map} of $L$. If $\mu_0(L)$ is injective then $L$ is said to be \emph{Brill-Noether-Petri}.  If all $L\in \Pic(C)$ are Petri then \emph{$C$ is Brill-Noether-Petri} (shortly BNP). 
 
 \end{defin}

\begin{rem}  The Brill-Noether number $\rho=g-(r+1)(g-d+r)$ is   non-negative for  every BNP invertible sheaf $L$ on $C$ such that $\deg L=d$ and $h^0(C,L)=r+1$.
We will almost always consider the case of base-point-free pencils on $C$. If $V\subset H^0(C,L)$ is one such then
\be\label{kermunot}
\ker\left\{\mu_{0,V}: V\otimes H^0(KL^{-1})\longrightarrow H^0(C, K)\right\}\,\,\cong \,\,H^0(C, KL^{-2})
\ee
Moreover the condition $\rho\geq 0$, for pencils, translates into: $2d\geq g+2$.
\end{rem}

The following theorem by Voisin \cite {Vois88} is at the centre of our paper.

\begin{thm}\label{T:voisin}
Assume that $\cliff(C) \ge 3$ and let $L$ be a general element of a component $S$ of  $W^1_{g-2}(C)$.  Then 
\begin{itemize}
\item[{\rm(i)}] $|L|$ is a primitive $g^1_{g-2}$, i.e. both $L$ and   $\omega_CL^{-1}$ are base-point-free.

\item[{\rm(ii)}] $|\omega_CL^{-1}|$ maps $C$ birationally onto a plane curve $\Gamma$ of degree $g$ having only double points as singularities.  

\item[{\rm(iii)}]  $L$ is BNP.
\end{itemize}

\end{thm}

\proof For (i) and (ii) we refer to \cite{Vois88}, Prop. II.0. 

(iii) If $L$ is not BNP then  $\omega_CL^{-2}$ is effective of degree two. Since $\dim(S)\ge g-6$ the rational map
\[
\xymatrix{
S \ar@{-->}[r]  & C^{(2)},&
L \ar@{|->}[r] & \omega_CL^{-2} }
\]
has positive dimensional fibres and this implies that $L^2$ has infinitely many distinct square roots,   which is impossible.  \qed

\bigskip

 From now on, in Part 1,  we will assume $\cliff(C) \ge 3$ and we will denote by $L$ an   element of $W^1_{g-2}(C)$ satisfying conditions (i), (ii) and (iii) of Theorem \ref{T:voisin}.   
 Note that, since $g \ge 11$, the condition $\cliff(C) \ge 3$ is equivalent to  the condition that $C$ does not have a $g^1_4$ (is not 4-gonal). 

\bigskip

In order to simplify the notation we will identify $C$ with its canonical model 
\[
\varphi_{\omega_C}(C) \subset \bP H^0(C,\omega_C)^\vee\ \cong\bP^{g-1}
\]
We will write $\bP$ instead of $\bP^{g-1}$ whenever no confusion is possible. To the sheaf $L$ there is associated, in a well known way, a cubic scroll $X\subset \bP$ containing $C$ (see e.g. \cite{Schr86}). Let us recall how this  is done.  We set:
\[
V :=  H^0(C,\omega_CL^{-1}) \cong \bC^3
\]
The Petri map
\[
\mu_0(L):V\otimes H^0(C,L) \lra H^0(C,\omega_C)=H^0(\bP,\cO(1))
\]
produces a $3\times 2$ matrix $M$ of linear forms on $\bP$ whose $2\times 2$ minors vanish on $C$.  
Explicitely let $x_0,x_1,x_2$ be  a basis for $H^0(\omega_CL^{-1})$. Let 
$y_1, y_2$ be basis of $H^0(L)$.  Since the Petri map is injective, the six elements $\omega_{ij}=x_i{y_j}$
are linearly independent  in $ H^0(\omega_C)$ and can be viewed as linear forms $X_{ij}$ in $\bP$.
Then 
\be\label{M}
M=\{X_{ij}\}\,,\qquad  {X_{ij}}_{|C}=\omega_{ij}=x_i{y_j}
\ee
The variety $X\subset \bP$ defined by the minors of $M$ is a cubic scroll of codimension two containing $C$. 

Choose a $(g-6)$-dimensional subspace $W \subset H^0(C,K_C)$ such that
\[
H^0(C,\omega_C) = \Im(\mu_0(L))\oplus W\cong (H^0(C,L)\otimes V) \oplus W
\]
 Consider on $\bP^1$ the vector bundle 
 \[
 \E := (V\otimes\cO_{\bP^1}(1))\oplus (W\otimes\cO_{\bP^1})
 \]
 and let $\wt X :=\bP\E$.  Then $X$ is the image of the morphism:
 \[
 \nu:=\varphi_{\cO_{\wt X}(1)}: \wt X \lra \bP^{g-1}
 \]
 It is a cone over the Segre variety $\bP^1\times \bP^2 \subset \bP^5$ with a $(g-7)$-dimensional vertex equal to
 $\bP W$.  In particular $X$ is arithmetically Cohen-Macaulay. One has 
 \[
 \nu^{-1}(\bP W) \cong \bP^1\times \bP W
 \]
while
\[
\nu_{|\wt X\setminus \nu^{-1}(\bP W)}: \wt X\setminus \nu^{-1}(\bP W)\lra X\setminus \bP W
\]
 is an isomorphism.
Let $\pi: \wt X \lra\bP^1$ be the natural projection and let 
\[
\wt X \supset \wt R = \pi^{-1}(p)\cong  \bP^{g-4}
\]
be any fibre of $\pi$.   Then $\wt R$ is a Cartier divisor in $\wt X$ which is mapped isomorphically onto its image $R:=\nu(\wt R)\subset X$. Note the following:
\medskip

\begin{enumerate}[(I)]\label{E:cliff1}
\item     $R \subset X$ is a Weil divisor containing $\bP W$ and not Cartier along $\bP W=\Sing(X)$.  
\item Since $R \cap C \in |L|$ and $|L|$ has no fixed points, we have $C \cap \bP W = \emptyset$. 
\item 
Identifying $C$ with $\wt C:=\nu^{-1}(C) \subset \wt X$ we have 
\[
\cO_{\wt C}(\wt R)= L = \cO_C(R)
\]
\end{enumerate}

\medskip

We shall denote by $H$ a hyperplane section of $X$ and we set
 \[
 \wt H := \nu^* H
 \]
 Of course we have:
 \[
 \cO_{\wt X}(\wt H) = \cO_{\wt X}(1), \quad  \cO_C(H) = \omega_C
 \]
 We will also need the following:
 
 \begin{lemma}\label{L:quad1}
 In the above situation we have:
 \begin{itemize}
\item[{\rm(i)}] All quadrics $Q \in H^0(\bP^{g-1},\I_{X/\bP}(2))$ have rank $\le 4$.

\item[{\rm(ii)}] For each point $p \in X\setminus \bP W$ there is a unique  (up to constant factor) non-zero quadric $Q_p \in H^0(\bP^{g-1},\I_{X/\bP}(2))$ that is  singular at $p$.

\item[{\rm(iii)}] If a quadric $Q\in H^0(\bP,\cO_{\bP}(2))$ contains $R\cup C$ for some   ruling $R\subset X$ and is singular at a point $p \in C$,  $p \notin R$, then $Q=Q_p \in H^0(\bP^{g-1},\I_{X/\bP}(2))$.
\end{itemize}

 \end{lemma}

 \proof (i) If $Q,Q' \in H^0(\bP,\I_{X/\bP}(2))$ are distinct then $Q \cap Q' = X\cup  Z$ where  $Z \cong \bP^{g-3}$. Then $Q$  has rank $\le 4$ because it contains a  $\bP^{g-3}$.
 
 (ii)  Projecting $X$ from $p$ we get a quadric $\Sigma\subset \bP^{g-2}$.  Then $Q_p$ is the cone projecting $\Sigma$ from $p$.
 
 (iii) The quadric $Q$ contains $R$ and is singular at $p$, thus it contains 
 $\langle p,R\rangle\cong \bP^{g-3}$. Hence $\rk(Q) \le 4$. By construction $Q$ is the quadric associated to the pencil $|L(p)|=|L|+p$, which is a $g^1_{g-1}$ (by definition, this is the union  of the codimension two spaces $\langle D\rangle$, as $D$ runs in  $g^1_{g-1}$).
  Then $Q$ contains all the spaces 
 $\langle p,R'\rangle$ as well, and therefore it contains $X$.  \qed

\vskip 0.5 cm
\section{A digression on blowing-ups}\label{S:digression}

In this section we fix some notation and recall some well-known facts. 

Let $S$ be a projective nonsingular algebraic surface, $P_0\in S$ and $c \ge 1$ an   integer. 
Let
\[
\sigma: \xymatrix{
\wt S=S_c\ar[r]^{\sigma_c}& \cdots \ar[r]^-{\sigma_2} &S_1  \ar[r]^-{\sigma_1}&S}
\]

be the composition of $c$ blowing-ups of $S$ with centers $P_0$ and     $c-1$ points $P_1,\dots, P_{c-1}$ belonging to consecutive successive infinitesimal neighborhoods  of $P_0$.  In other words, $P_j$ belongs to the $(-1)$-curve on $S_j$ coming from the blowing-up of $P_{j-1}\in S_{j-1}$,  $j=1,\dots, c-1$. If $j \ge 2$ we assume that $P_j$ does not also belong to the $(-2)$-curve  coming from the blowing up of $P_{j-2}$.
\bigskip

NOTATION: We let $e_1,\dots, e_c \subset \wt S$ be the components of the tree of exceptional rational curves arising from the $c$ blowing-ups, where $e_j$ is the rational curve coming from the $j$-th blowing-up. Therefore:
\[
e_c^2=-1, \quad e_j^2=-2, \ j=1,\dots, c-1,\quad e_j\cdot e_{j+1}=1, \ j=1,\dots, c-1
\]
and $e_j\cdot e_h=0$ otherwise.   
 Moreover we let
\[
E_j:= \sum_{h=j}^ce_h
\]
We have:
\[
E_j^2=-1, \quad E_i\cdot E_j=0, \ i\ne j
\]
\bigskip

With this notation we have:
\begin{align}
K_{\wt S} &= \sigma^*K_S+ \sum_j E_j \notag 
\end{align}
Note also that
\begin{equation}\label{E:eij}
E_j\cdot e_h = \begin{cases}1&\text{if $j=h+1$}\\-1&\txt{if $j=h$}\\ 0&\text{otherwise}\end{cases}
\end{equation}

\bigskip

Let $D\subset S$ be an integral curve. Recall that  $P\in D$ is called an \emph{$A_n$-singularity}, $n \ge 1$, if a local equation of $D$ at $P$ is analytically equivalent to $y^2+x^{n+1}=0$. All $A_n$-singularities are double points and conversely every double point is an $A_n$-singularity for a unique $n$. The delta-invariant of an $A_n$ singularity is:
\[
\delta = \left[\frac{n+1}{2}\right]
\]
 (\cite{GLS07}, p. 206). 
Let $\mathbf{c}\subset \cO_{D,P}$ be the conductor ideal, i.e. the annihilator of $\overline{\cO}_{D,P}/\cO_{D,P}$, where $\overline{\cO}_{D,P}$ denotes the integral closure of $\cO_{D,P}$. Then $\dim_{\bC}\cO_{D,P}/\mathbf{c}=\delta$. Let $\mathbf{A}\subset\cO_{S,P}$ be the inverse image of $\mathbf{c}$ under $\cO_{S,P}\lra \cO_{D,P}$ ($\mathbf{A}$ is the \emph{local adjoint ideal} of $D$ at $P$). Since it has the same colength  $\delta$ in $\cO_{S,P}$ as $\mathbf{c}\subset \cO_{D,P}$ and contains the jacobian ideal of $D$ at $P$ \cite{DH88}, the local adjoint ideal of an $A_n$-singularity is analytically equivalent to 
\[
(x^\delta,y)=(x^{\left[\frac{n+1}{2}\right]},y)
\]
In particular it defines a curvilinear scheme of colength $\delta$.
 An $A_n$-singularity of $D$ at $P$ can be resolved by precisely $\delta$ blowing-ups 
(\cite{GLS07}, Prop. 3.34).  Therefore, after performing the appropriate blowing-ups of $S$ we end-up with a diagram
\[
\xymatrix{
 \wt D \ar[d]\ar@{^(->}[r]& \wt S \ar[d]^-\sigma \\
D \ar@{^(->}[r]& S}
\]
where $\wt D$  is the proper transform of $D$ and $\wt S$  contains a tree of rational curves exactly analogous to the one described above, with $c$ replaced by $\delta$.  Moreover:
\[
\wt D = \sigma^*D-2\sum_jE_j
\]
and
\[
\wt D\cdot e_\delta > 0, \quad \wt D \cdot e_j=0,\ j=1,\dots, \delta-1
\]
The last assertions follow from the fact that the total transform of an $A_n$-singularity, $n \ge 3$, is a $D_{n+1}$-singularity (see \cite{BPVdV84}, table on p. 65).
Finally observe that   if $\delta \ge 2$ the union
\[
\bigcup_{j=1}^{\delta-1} e_j
\]
is a Hirzebruch-Jung configuration and therefore it can be contracted to a rational double point (\cite{BPVdV84},  Prop. 3.1 p. 74).


 \vskip 0.5 cm
\section{The surface $P$}\label{surface-P}

The plane curve $\Gamma := \varphi_{\omega_CL^{-1}}(C) \subset \bP^2$ is   irreducible    of degree $g$ and has only double points $P_1,\dots, P_k$  as singularities, by Theorem \ref{T:voisin}.  For $i=1,\dots, k$ the point $P_i$ is an $(A_{n_i})$-singularity for $\Gamma$, for some $n_i$.  
Let $\delta_i=\left[\frac{n_i+1}{2}\right]$  be the corresponding local delta-invariant. Then:
\[
\sum_i\delta_i =:\delta = {g-1\choose 2}-g
\]
By what we have seen in \S \ref{S:digression}  the adjoint ideal sheaf $Adj(\Gamma)\subset \cO_{\bP^2}$ defines a  curvilinear scheme.  Let $Y\subset \bP^2$ be a general adjoint curve to $\Gamma$ of degree $g-4$.  Since $|L|$ is a base point free pencil, $Y$ is nonsingular at $P_1,\dots, P_k$. 
We have a diagram:
\[
 \xymatrix{
 C \ar@{^(->}[r]\ar[d]_{\varphi_{KL^{-1}}}&\wt P \ar[d]^-\sigma \\
\Gamma \ar@{^(->}[r] & \bP^2 }
\]
where $\sigma$ is a sequence of $\delta$ blowing-ups and $C$ is the proper transform of $\Gamma$. 
We will denote by $e_{i1},\dots,e_{i\delta_i}$ the components of $\sigma^{-1}(P_i)$ and by 
\[
E_{ij} = \sum_{h=j}^{\delta_i} e_{ih}
\]

For the general theory of adjoint curves we refer to Szpiro's book \cite{Szp79}.

  \begin{prop}\label{P:copp}
  The line bundle $\cO_{\wt P}(K_{\wt P}+C)$ is globally generated and maps $\wt P$ birationally onto a  surface $P\subset \bP$ obtained by contracting all the $(-2)$-curves of $\wt P$. In particular $P$ is normal and has at most  rational double points. 
If $\Gamma$ is nodal then $P \cong \wt P$ is nonsingular. Furthermore $P\subset X$.
  \end{prop}
  
  \proof
   Let 
\[
E = \sum E_{ij}
\]
and let $\wt Y\subset \wt P$ be the proper transform of $Y$. We have
  \begin{align}
  K_{\wt P}&= \sigma^*\cO_{\bP^2}(-3)(E) \notag \\ 
C&=\sigma^*\Gamma-2E= \sigma^*\cO_{\bP^2}(g)(-2E)\notag \\
K_{\wt P}+C &\sim \sigma^*\cO_{\bP^2}(g-3)(-E)\notag \\
\wt Y &\sim \sigma^*\cO_{\bP^2}(g-4)(-E)\notag 
\end{align}
Let   $\Lambda\subset \bP^2$ be a line and let $\lambda\subset \wt P$ be its proper transform in $\wt P$. 
Assume first that $\Lambda$ does not contain any of the points $P_1,\dots, P_k$.  Since $|\omega_CL^{-1}|$ is a complete  $g^2_g$  the adjoint ideal to $\Gamma$ imposes independent conditions to the curves of degree $\ge g-4$. This is an exercise in \cite{ACGH84}, p. 57, in the nodal case.  In the general case this  follows  from the general theory of adjoints as explained in  \cite{Szp79}, p. 40-42.  As a result 
the  exact sequence:
\[
\xymatrix{
0 \ar[r]&\sigma^*\cO_{\bP^2}(g-4)(-E)\ar[r] & \sigma^*\cO_{\bP^2}(g-3)(-E) \ar[r] & \cO_\lambda(g-3) \ar[r] &0}
\]
is exact on global sections, and therefore $|K_{\wt P}+C|$ embeds $\lambda$.  
Assume now that $\Lambda$ contains   $P_1$ and none of the $P_i$'s for $i \ge 2$. We have:
\[
\cO_{\wt P}(\lambda) = \sigma^*\cO_{\bP^2}(1)(-\sum_{j\le j_0}E_{1j})
\] 
for some $j_0 \ge 1$, and
\[
E\cdot \sum_{j\le j_0}E_{1j} = -j_0
\]
Consider the exact sequence:
\[
\xymatrix{
0 \ar[r]&\sigma^*\cO_{\bP^2}(g-4)(-E+\sum_{j\le j_0}E_{1j})\ar[r] & \sigma^*\cO_{\bP^2}(g-3)(-E) \ar[r] & \cO_\lambda(g-3-j_0) \ar[r] &0}
\]
As $C\cdot\lambda\geq 0$ ($C$ is irreducible), and $C\cdot\lambda=g-2j_0$, we get $g\geq 2j_0$. On the other hand $g\geq 11$, so that $g-3-j_0>0$. 
Since $h^0(\wt P, \sigma^*\cO_{\bP^2}(g-4)(-E+\sum_{j\le j_0}E_{1j}))=2+j_0$ the above  sequence is exact on global sections.
A similar argument can be given for any other line $\Lambda \subset \bP^2$, no matter how many of the $P_i$'s it contains.
This implies that $|K_{\wt P}+C|$ is very ample on the  proper transform $\lambda$ of any line $\Lambda\subset \bP^2$. Therefore $|K_{\wt P}+C|$ is base-point-free and very ample on $\wt P \setminus \bigcup e_{ij}$. Let us prove it is also base-point-free on $\bigcup e_{ij}$.
\vskip 0.3 cm
{\it $|K_{\wt P}+C|$ has no base point on $e_{i,\delta_i}$} : set $C\cdot e_{i,\delta_i}=q_1+q_2$.  If there were a base point for $|K_{\wt P}+C|$ this could not lie on $C$, as $|K_{\wt P}+C|$ cuts out on $C$ the canonical series. But then, since 
$(K_{\wt P}+C)\cdot e_{i,\delta_i}=1$, any element of $|K_{\wt P}+C|$ passing through $q_1$ would pass through $q_2$ (indeed would contain $e_{i,\delta_i}$ as a component)
which is absurd, since $C$ is non-hyperelliptic.

\vskip 0.3 cm
{\it $|K_{\wt P}+C|$ has no base point on $e_{i,j}$ with $j<\delta_i$}:  We know that $(K_{\wt P}+C)\cdot e_{i,j}=0$, if $j<\delta_i$,  so that $|K_{\wt P}+C|$ has a base point on $e_{i,j}$ only if  $e_{i,j}$ is a fixed component of 
$|K_{\wt P}+C|$. But then $|K_{\wt P}+C|$ has a fixed point on $e_{i,j+1}$ and so on, until one reaches $e_{i,\delta_i}$,
and a contradiction.
\vskip 0.3 cm

It remains to check how $|K_{\wt P}+C|$ restricts to the curves $E_{ij}$. Using \eqref{E:eij} we deduce that
\[
(K_{\wt P}+C)\cdot e_{ij}= \begin{cases} 1 & \text{if $j=\delta_i$} \\ 0&\text{if $1\le j\le \delta_i-1$}\end{cases}
\]
Therefore $|K_{\wt P}+C|$ contracts the Hirzebruch-Jung chains $e_{i1}+\cdots+ e_{i\delta_i-1}$ of $(-2)$-curves to rational double  points  and maps the $(-1)$-curves $e_{i\delta_i}$ isomorphically onto lines.
 Then the morphism 
\[
\Phi:  \wt P\lra \bP^{g-1}
\]
defined by the linear system $|K_{\wt P}+ C|$   maps $\wt P$    onto a surface $P$ containing $C$ and having the asserted properties.
Let us finally show that $P$ is contained in $X$. Let $x_0,x_1,x_2$ be a basis of  $H^0(\omega_CL^{-1})$ and recall the matrix $M$ introduced in (\ref{M}).
Let 
$\phi_1=\phi_1(x_0,x_1,x_2)$, $\phi_2=\phi_2(x_0,x_1,x_2)$ be degree-$(g-4)$ adjoint polynomials such that
${\phi_1}_{|C}=y_1, {\phi_2}_{|C}=y_2$ is 
a basis of $H^0(L)$.  Under the morphism $\Phi: \wt P\to \bP$, we get $\Phi^*(X_{ij})=x_i\phi_j$. This means that
the matrix $M$ has rank one on $P$, or equivalently that $P$ is contained in $X$.
This
 concludes the proof. \qed 
  
\bigskip

For each ruling $R \subset X$ we will set
\[
Y_R := R \cap P
\]
and 
\[
\wt Y_R := \Phi^{-1}(Y_R) \subset \wt P
\]
Note that the curves $\wt Y_R$ are the proper transforms of the adjoint curves to $\Gamma$ of degree $g-4$.  In particular
\[
\cO_C(\wt Y_R) = L
\]

\begin{lemma}\label{L:quad2}
\[
h^0(X,\I_{C\cup R/X}(2H)) \ge h^0(X,\I_{P\cup R/X}(2H)) \ge g-5
\]
\end{lemma}

\proof  The first inequality is obvious. From the exact sequence:
\[
\xymatrix{
0 \ar[r] & \I_{X/\bP}(2H) \ar[r] & \I_{R/\bP}(2H) \ar[r] & \I_{R/X}(2H) \ar[r]&0}
\]
we deduce:
\begin{align}
h^0(X,\I_{R/X}(2H)) & \ge h^0(\bP,\I_{R/\bP}(2)) - h^0(\bP,\I_{X/\bP}(2)) \notag \\
&= (3g-3) - 3 = 3g-6 \notag
\end{align}
Now consider the exact sequence
\[
\xymatrix{
0 \ar[r]& \I_{P\cup R/X}(2H) \ar[r]& \I_{R/X}(2H) \ar[r] & \I_{Y_R/P}(2H) \ar[r] &0}
\]
which implies:
\[
h^0(X,\I_{P\cup R/X}(2H)) \ge 3g-6 - h^0(P,\I_{Y_R/P}(2H))
\]
Finally observe that  
\begin{align}
h^0(P,\I_{Y_R/P}(2H)) &\le h^0(\wt P,\cO_{\wt P}(2H-\wt Y_R)) \notag \\
&\le h^0(C,\omega^2_CL^{-1}) \notag \\
&=2g-1 \notag
\end{align}
The first inequality is obvious, and the second inequality follows from the exact sequence:
\[
\xymatrix{
0 \ar[r]& \cO_{\wt P}(2H-\wt Y_R-C) \ar[r] & \cO_{\wt P}(2H- \wt Y_R) \ar[r]&\omega^2_CL^{-1}\ar[r]&0}
\]
and from the fact that $H^0(\wt P,\cO_{\wt P}(2H-\wt Y_R-C))=0$. \qed

 \begin{prop}\label{P:nonsingP}
 The surface $P$ is nonsingular along $C$.
 \end{prop}
 
 \proof   Keeping the notation of Proposition \ref{P:copp}, the curve $C\subset \wt P$ meets the $(-1)$-curves $e_{i\delta_i}$ and does not meet the other $e_{ij}$'s because $C\cup e_{i\delta_i}\cup e_{i\delta_i-1}$ is the total transform of a $D_4$ or $D_5$ singularity (\cite{BPVdV84}, table on p. 65).  This implies that $C$ does not contain any of the singular points of $P$.
 
\qed

\vskip 0.5 cm
\section{Triviality of a (restricted)  conormal sheaf}\label{conormal}
 If $A\subset B$ we shall adopt the standard notation $N^\vee_{A/B}$ for the conormal sheaf $\I_{A/B}/\I_{A/B}^2$ of $A$ in $B$. 

\begin{lemma}\label{L:conor1}
$N^\vee_{P/X|C}=N^\vee_{P/X}\otimes\cO_C$ is locally free and 
\[
c_1(N^\vee_{P/X|C}(2H-L))=0
\]
\end{lemma}

 \proof By Proposition \ref{P:nonsingP}, the conormal sheaf 
 $N^\vee_{P/X}$ is locally free of rank $g-5$ along $C$. Consider the exact sequence of locally free sheaves on $C$:
 \[
 \xymatrix{
 0 \ar[r]&N^\vee_{C/X}(2H-L) \ar[r]& \Omega^1_{X|C}(2H-L) \ar[r] & \omega_C(2H-L) \ar[r] & 0}
 \]
 We have:
 \begin{align}
 c_1(N^\vee_{C/X}(2H-L)) &=c_1(\Omega^1_{X|C}(2H-L))- c_1(\omega_C(2H-L))\notag \\
&= c_1(\Omega^1_{\wt X|C}(2H-L))- c_1(\omega_C(2H-L))\notag \\
&= -(g-3)H+L+(g-3)(2H-L)-(3H-L) \notag\\
&= (g-6)H - (g-5)L  \notag \\
&= (g-6)\omega_C - (g-5)L \notag
\end{align}
On the other hand, since $P$ is nonsingular along $C$, we have the exact sequence on $C$:
\[
\xymatrix{
0\ar[r]&N^\vee_{P/X|C}(2H-L) \ar[r] & N^\vee_{C/X}(2H-L) \ar[r] & N^\vee_{C/P}(2H-L) \ar[r] & 0
}
\]
  and therefore:
  \begin{align}
  c_1(N^\vee_{P/X|C}(2H-L))&=c_1(N^\vee_{C/X} (2H-L))-c_1(N^\vee_{C/P} (2H-L))\notag \\
  &= (g-6)\omega_C-(g-5)L- (2\omega_C-L-\cO_C(C))    \notag\\
  &=(g-8)\omega_C-(g-6)L+\cO_C(C) \notag
  \end{align}
 We can rewrite this  class as the restriction of a class on $\wt P$ by letting 
 \[
 \ell := \sigma^*\cO_{\bP^2}(1)
 \]
 With this notation we have:
 \[
 \omega_C = \cO_C((g-3)\ell-E), \quad L=\cO_C((g-4)\ell-E), \quad 
 \cO_C(C)= \cO_C(g\ell-2E)
 \]
where $E$ has the same meaning as in the proof  of Proposition \ref{P:copp}. 
 Therefore the class $c_1(N^\vee_{P/X|C}(2H-L))$ is the restriction to $C$ of the class on $\wt P$:
 \[
 (g-8)((g-3)\ell-E) - (g-6)((g-4)\ell-E) + (g\ell-2E) 
= 0
\]
 \qed

 \bigskip
 
 Now we can prove the following triviality result.

 \begin{prop}\label{P:trivial}
  The restriction homomorphism:
  \[
  \I_{P\cup R/X}(2H) \lra N^\vee_{P/X|C}(2H-L)
  \]
  induces   an isomorphism:
  \[
  \bar\alpha: \xymatrix{
  H^0(X,\I_{P\cup R/X}(2H))\otimes\cO_C \ar[r]^-{\cong}& N^\vee_{P/X|C}(2H-L)}
\]
  \end{prop}
  
  \proof  By Proposition \ref{P:nonsingP},
the sheaf $N^\vee_{P/X|C}(2H-L)
$ is locally free of rank $g-5$. 
 Choose $p\in C$ but $p\notin R$.
Consider the vector space  homomorphism:
\[
\bar\alpha_p:H^0(X,\I_{P\cup R/X}(2H)) \lra N^\vee_{P/X|C}(2H-L)_p\otimes \bC
\]
 The homomorphism $\bar\alpha_p$ fits into the following commutative and exact diagram:
\[
\xymatrix{
0\ar[r]&H^0(\bP, \I_{X /\bP}(2H))\ar[r]\ar[d]^{\gamma_p}&H^0(\bP, \I_{P \cup R /\bP}(2H))\ar[r]\ar[d]^{\beta_p}&H^0(X, \I_{P \cup R / X}(2H))\ar[d]^{\bar\alpha_p}\ar[r]&0\\
0\ar[r]& N^{\vee}_{X/ \bP}(2H)_p\otimes \bC\ar[r]\ar@{=}[d]& 
N^{\vee}_{P/ \bP}(2H-R)_p\otimes \bC\ar[r]& N^{\vee}_{P/X}(2H-R)_p\otimes \bC\ar[r]&0 \\
&N^{\vee}_{X/ \bP}(2H-R)_p\otimes \bC
}
\]
The second row  is  exact  since  $N^{\vee}_{P/X}$ is locally free along $C$. The first row is exact because   $H^1(\bP,\I_{X/\bP}(2))=0$.     
  The homomorphism $\gamma_p$ is surjective since
$p$ is a smooth point of $X$ and $X$ is cut out by quadrics. It follows that there is a surjection
$$
\rho:  \ker(\beta_p) \longrightarrow \ker(\bar\alpha_p)
 $$
Now let  $Q \in \ker(\bar\alpha_p)$ and lift $Q$ to  $Q' \in \ker(\beta_p)$.
This means that $Q'$ is singular at $p$ and contains $R\cup C$.  Then $Q'$ contains $X$, by Lemma \ref{L:quad1} (iii), and therefore 
$Q=\rho(Q')=0$.  Thus $\bar\alpha_p$ is injective.

 Recall now (Lemma \ref{L:quad2}) that $h^0(X,\I_{P\cup R/X}(2H))\ge g-5$. Therefore $\bar\alpha_p$ is an isomorphism. 
 Thus $\bar\alpha$ is a morphism between a trivial bundle of rank $g-5$ and  a locally free rank-$(g-5)$ bundle with trivial Chern class. Since, for a point $p\in C$ the homomorphism $\bar\alpha_p$ is an isomorphism, $\bar\alpha$ is an isomorphism as well.
 
 \qed

\vskip 0.5 cm
\section{The square of the ideal sheaf of $C\subset X$}\label{square1}

The purpose of this section is to prove the following Proposition.

\begin{prop}\label{P:square1}
\[
H^1(X,\I^2_{C/X}(kH))=0
\]
for all $k \ge 3$.
\end{prop}

\proof
We present the proof in the case when $k=3$, leaving  to the reader the straightforward  generalisation to the case when $k\geq 4$. Consider the following diagram:
\be\label{diag-P-X}
\xymatrix{&0\ar[d]&0\ar[d]&0\ar[d]\\
0\ar[r]& \I \ar[r]\ar[d]& \I^2_{C/X}(3H) \ar[r]\ar[d]& \cO_P(3H-2C) \ar[r]\ar[d]&0 \\
0\ar[r]& \I_{P/X}(3H) \ar[r] \ar[d]_-\alpha & \I_{C/X}(3H)\ar[r]\ar[d]_-\beta& \cO_P(3H-C)\ar[r]\ar[d]_-\gamma&0 \\
 0\ar[r]& N^\vee_{P/X|C}(3H) \ar[r]\ar[d]& N^\vee_{C/X}(3H) \ar[r]\ar[d]& \cO_C(3H-C)\ar[r]\ar[d]&0 \\
 &0&0&0 }
\ee
 where $\I \subset \cO_X(3H)$ is a sheaf defined by the diagram.
 \medskip
 
 \emph{Claim:} \begin{itemize}
\item[(a)] $H^1(X,\I_{C/X}(3H))=0$
\item[(b)] $H^1(P,\cO_P(3H-2C))=0$
\item[(c)] $H^1(X,\I_{P/X}(3H))=0$
\end{itemize}

\medskip

 (a) This  is a consequence of  the fact that $X$ is arithmetically Cohen-Macaulay and $C$ is projectively normal.

(b)  Let $\wt H = \Phi^*H$.  Then we have:
\begin{align}
H^1(\wt P,\cO_{\wt P}(3\wt H-2C)) &= H^1(\wt P, 3K_{\wt P}+C) \notag\\
&= H^1(\wt P,\sigma^*\cO_{\bP^2}(g-9)+E)\notag \\
&= H^1(\bP^2,\cO_{\bP^2}(g-9))=0
\end{align}
  because $\sigma_*\cO(E)= \cO_{\bP^2}$ and $R^1\sigma_*\cO(E)=0$. Moreover we have
\[
H^1(P,\cO_P(3H-2C))= H^1(\wt P,\cO_{\wt P}(3\wt H-2C))
\]
In fact, by Proposition \ref{P:copp}, we get  $\Phi_*\cO_{\wt P}=\cO_P$ and $R^1\Phi_*\cO_{\wt P}=0$.

(c) Consider the exact sequences:

$$
\xymatrix{
&0\ar[d]\\
0 \ar[r]&\I_{X/\bP^{g-1}}(k) \ar[d]\ar[r]&\I_{P/\bP^{g-1}}(k) \ar[r]&\I_{P/X}(kH)\ar[r]&0 \\
&\cO_{\bP^{g-1}}(k)\ar[d]\\
&\cO_X(kH)\ar[d]\\
&0}
$$
From the vertical one and the CM property of $X$ we deduce
\vskip 0.3 cm

$(c_1)$\  $H^1(\bP^{g-1},\I_{X/\bP^{g-1}}(k)) = H^2(\bP^{g-1},\I_{X/\bP^{g-1}}(k))=0$ for all $k$.

From the horizontal one and from $(c_1)$ we deduce:
\vskip 0.3 cm

$(c_2)$ \  $H^1(X,\I_{P/X}(kH)) \cong H^1(\bP^{g-1},\I_{P/\bP^{g-1}}(k))$ for all $k$.

\vskip 0.3 cm

Therefore for proving (c) it suffices to prove that $\I_{P/\bP^{g-1}}$ is 3-regular,  i.e. that
\[
H^i(\bP^{g-1},\I_{P/\bP^{g-1}}(3-i))=0 \ \ \text{for all $i > 0$}
\]
(see e.g. \cite{Sern06},  Prop. 4.1.1).
\vskip 0.3 cm

$i=1$. \ The second row of \eqref{diag-P-X} 
tells us that:
\begin{align}
h^0(X,\I_{P/X}(2H)) &\ge h^0(X,\I_{C/X}(2H))-h^0(\cO_P(2H-C)) \notag \\
&={g-2\choose 2}-3-{g-4\choose 2} = 2g-10\notag
\end{align}
while the first column of \eqref{diag-P-X} 
gives
\[
h^0(X,\I_{P/X}(2H)) \le h^0(C,N^\vee_{P/X|C}(2H))+ h^0(X,\I(-1)) = 2g-10+0 = 2g-10
\]
because $N^\vee_{P/X|C}(2H)\cong L^{\oplus g-5}$ and because $h^0(X,\I^2_{C/X}(2H))=0$, since a quadric cannot be singular along the \emph{non-degenerate} $C$.  Therefore  $h^0(X,\I_{P/X}(2H))=2g-10$.   Since $H^1(X,\I_{C/X}(2H))=0$, we obtain $H^1(X,\I_{P/X}(2H))=0$. Finally,
by $(c_2)$ we have:

\[
H^1(\bP^{g-1},\I_{P/\bP^{g-1}}(2))= H^1(X,\I_{P/X}(2H))=0
\]
  This takes care of the case $i=1$. 
The remaining cases $i \ge 2$ follow easily from the exact sequences:
\[
\xymatrix{
0\ar[r]&\I_{P/\bP^{g-1}}(3-i) \ar[r]&\cO_{\bP^{g-1}}(3-i)\ar[r]&\cO_P((3-i)H) \ar[r]&0}
\]

\emph{The Claim is proved.}

\bigskip

From (a)   it follows that  the Proposition is equivalent to the surjectivity of  $H^0(\beta)$. From (b) and (c)  we deduce that $H^0(\beta)$ is surjective  if $H^0(\alpha)$ is surjective, where:
\[
H^0(\alpha): H^0(X,\I_{P/X}(3H)) \lra H^0(N^\vee_{P/X|C}(3H))\cong
H^0(X,\I_{P\cup R/X}(2H))\otimes H^0(C,\omega_CL)\,.
\]
Here we use  Proposition \ref{P:trivial}  to get the isomorphism:
\[
N^\vee_{P/X|C}(3H) \cong H^0(X,\I_{P\cup R/X}(2H))\otimes \omega_CL
\]
Recall that the morphism $\nu: \wt X \lra X$ induces an isomorphism
\[
\nu_{|\wt U}:  \wt U \lra U
\]
where $U = X \setminus \bP W$ and  $\wt U = \wt X \setminus \nu^{-1}(\bP W)$.
\medskip

\emph{Claim:}\begin{itemize}
\item[(d)] We have isomorphisms:
\[
H^0(U,\cO_X(H+R)) \cong H^0(\wt U, \cO_{\wt X}(\wt H+\wt R))\cong H^0(C,\omega_CL)
\]
\end{itemize}

\medskip

\emph{Proof of (d).}  The first isomorphism is obvious. The second isomorphism is given by the composition:
\[
H^0(\wt U, \cO_{\wt X}(\wt H+\wt R)) \lra H^0(\wt X,\cO_{\wt X}(\wt H+\wt R))\lra
H^0(C,\omega_CL)
\]
Here the first arrow is the inverse of the restriction and is an isomorphism because $\wt X\setminus \wt U$ has high codimension. 
The second arrow too is an isomorphism. This  follows from the exact sequence:
\[
\xymatrix{
0 \ar[r] & \I_{C/\wt X}(\wt H+\wt R) \ar[r] & \cO_{\wt X}(\wt H+\wt R) \ar[r]& \omega_CL\ar[r]&0}
\]
and the fact that $h^0(\wt X,\cO_{\wt X}(\wt H+\wt R))=2g-3= h^0(C,\omega_CL)$, a computation that follows  from the sequence:
\[
\xymatrix{
0 \ar[r] & \cO_{\wt X}(\wt H) \ar[r] & \cO_{\wt X}(\wt H+\wt R) \ar[r]& \cO_{\wt R}(\wt H)\ar[r]&0}
\]
and the fact that  $H^0(\wt X,\I_{C/\wt X}(\wt H+\wt R)) = 0$. This last  equality follows from the exact sequence
\[
\xymatrix{
0 \ar[r] & \I_{C/\wt X}(\wt H) \ar[r] & \I_{C/\wt X}(\wt H+\wt R)  \ar[r]& \I_{C \cap \wt R/ \wt R}(\wt H) \ar[r]&0}
\]
\emph{The Claim is proved.}

\bigskip

From the Claim we deduce the following commutative diagram:
\[
\xymatrix{
H^0(X,\I_{P\cup R/X}(2H))\otimes H^0(U,\cO_X(H+R))\ar[r]\ar[dr]\ar[d]_\cong& H^0(X,\I_{P/X}(3H)) \ar[d]^-{H^0(\alpha)} \\
H^0(C,N^\vee_{P/X|C}(2H-R))\otimes H^0(C,\omega_CL) \ar@{=}[r]&H^0(C,N^\vee_{P/X|C}(3H))}
\]
where the left vertical arrow is an isomorphism.  This implies the surjectivity of $H^0(\alpha)$. \qed

\vskip 0.5 cm
\section{The square of the ideal sheaf of $X\subset \bP^{g-1}$}\label{square2}

We begin  this section with two preliminary lemmas.

\begin{lemma}\label{L:square1}
Let $r\ge 1$ and let $Z:=\bP^1\times \bP^r \subset \bP^{2r+1}$  by the Segre embedding. Then
\[
H^1(Z,\I^2_{Z/\bP^{2r+1}}(k))=0 
\]
for all $k$.
\end{lemma}

\proof  For $k\ne 2$ the result is already proved in \cite{Wah97}, (1.3.2) and Proposition 4.4.  Moreover we know that:
\[
H^1(Z,\I^2_{Z/\bP^{2r+1}}(2)) = \ker(\Phi_H)
\]
where 
\[
\Phi_H: \wedge^2 H^0(Z,\cO_Z(H)) \lra H^0(Z,\Omega_Z^1(2H))
\]
is the gaussian map of $H=\cO_Z(1)$   (\cite{Wah97}, Proposition 1.8).  Therefore it suffices to prove that
\emph{$\Phi_H$ is an isomorphism.}

\medskip

Set  $E=H^0(\bP^1,\cO_{\bP^1}(1))$,  $F=H^0(\bP^r,\cO_{\bP^r}(1))$ and consider the projections:
\[
\xymatrix{
Z \ar[r]^-{\pi_2}\ar[d]_-{\pi_1}& \bP^r \\
\bP^1}
\]
Then we have:
\[
H^0(Z,\cO_Z(H)) = E\otimes F 
\]
and a direct sum decomposition:
\begin{equation}\label{E:gauss1}
\wedge^2(E\otimes F) = \left(\wedge^2E\otimes S^2F\right)
\oplus \left(S^2E\otimes \wedge^2F\right)
\end{equation}
 On the other hand we have:
 \[
 \Omega^1_Z (2H) = (\pi_1^*\Omega^1_{\bP^1})(2H)\oplus (\pi_2^*\Omega^1_{\bP^r})(2H)
 \]
 and from K\" unneth formula we get 
 \begin{equation}\label{E:gauss2}
 H^0(Z,\Omega^1_Z (2H)) = \left[H^0(\bP^1,\Omega^1_{\bP^1}(2))\otimes S^2 F \right]
 \oplus  \left[S^2E\otimes H^0(\bP^r,\Omega^1_{\bP^r}(2))\right]
 \end{equation}
 Comparing \eqref{E:gauss1} and \eqref{E:gauss2}  we see that the homomorphism
 \[
 \Phi_H:\wedge^2(E\otimes F) \lra H^0(Z,\Omega^1_Z (2H))
 \]
 decomposes as the sum of the two homomorphisms:
 \[
 [\Phi_{\cO_{\bP^1}(1)}]\otimes 1_{S^2F}: \wedge^2E\otimes S^2F \lra H^0(\bP^1,\Omega^1_{\bP^1}(2))
 \otimes S^2F
 \]
 and
 \[
 1_{S^2E}\otimes \Phi_{\cO_{\bP^r(1)}}:S^2E\otimes \wedge^2F \lra 
 S^2E\otimes H^0(\bP^r,\Omega^1_{\bP^r}(2))
 \]
 Both these homomorphisms are isomorphisms, and therefore so is $\Phi_H$. \qed
 
 \begin{lemma}\label{L:square2}
 Let $Y \subset \bP^{r+1}$, $r \ge 2$, be non-degenerate and linearly normal and  consider  a hyperplane section $Z=H\cap Y\subset H\cong\bP^r$. If $H^1(Z,\I^2_{Z/\bP^r}(k))=0$ for all $k$ then 
 $H^1(Y,\I^2_{Y/\bP^{r+1}}(k))=0$ for all $k$.
 \end{lemma}
 
 \proof We have exact sequences:
 \begin{equation}\label{E:gauss3}
 \xymatrix{
 0\ar[r]&\I^2_{Y/\bP^{r+1}}(k-1) \ar[r]& \I^2_{Y/\bP^{r+1}}(k) \ar[r] & \I^2_{Z/\bP^r}(k) \ar[r]&0 }
 \end{equation}
 and
 \begin{equation}\label{E:gauss4}
 \xymatrix{
 0\ar[r]&\I^2_{Y/\bP^{r+1}}(k) \ar[r] &\I_{Y/\bP^{r+1}}(k) \ar[r] & N^\vee_{Y/\bP^{r+1}}(k)\ar[r]&0}
 \end{equation}
Since 
\[
H^0(Y,N^\vee_{Y/\bP^{r+1}}(k))= 0 = H^1(\bP^{r+1},\I_{Y/\bP^{r+1}}(k))
\]
for $k \ll 0$, from   \eqref{E:gauss4} we deduce that $H^1(\bP^{r+1},\I^2_{Y/\bP^{r+1}}(k))=0$ for $k\ll 0$. 
On the other hand the hypothesis and \eqref{E:gauss3} imply that
\[
H^1(\bP^{r+1},\I^2_{Y/\bP^{r+1}}(k-1)) \lra H^1(\bP^{r+1},\I^2_{Y/\bP^{r+1}}(k))
\]
is a surjection for all $k$. This concludes the proof. \qed

\bigskip

Thes two  lemmas together imply the following:

\begin{prop}\label{P:square2}
\[
H^1(\bP,\I^2_{X/\bP}(k))=0
\]
for all $k$.
\end{prop}

\proof Recalling that $X$ is a cone over the smooth scroll $\bP^1\times\bP^2\subset \bP^5$ the result is obtained by applying repeatedly Lemma \ref{L:square2} and using Lemma \ref{L:square1} at the first step.   \qed

\bigskip

In \cite{Wah97} a similar result is proved for \emph{smooth} scrolls.  In our case  we cannot apply  Proposition 4.6 of \cite{Wah97} directly.

\vskip 0.5 cm
\section{Proof of Theorem \ref{T:main}}

The proof will follow from an adaptation to our case  of Proposition 3.2 of \cite{Wah97}.

Consider the following commutative diagram:
\[
\xymatrix{
0\ar[r]& H^0(\bP,\I_{X/\bP}(k))\ar[r]\ar[d]^-{\alpha_k}&H^0(\bP,\I_{C/\bP}(k))\ar[r]\ar[dd]^-{\gamma_k}&
H^0(X,\I_{C/X}(k))\ar[r]\ar[dd]^-{\delta_k}&0 \\
&H^0(X,N^\vee_{X/\bP}(kH))\ar[d]^-{\beta_k} \\
0\ar[r] & H^0(C,N^\vee_{X/\bP}(kH)\otimes\cO_C) \ar[r]& H^0(C,N^\vee_{C/\bP}(k))\ar[r]& H^0(C,N^\vee_{C/X}(k))}
\]
Since $H^1(\bP,\I_{C/\bP}(k))=0$ for $k \ge 3$ it will suffice to prove that $\gamma_k$ is surjective for all $k \ge 3$. From Propositions \ref{P:square1} and \ref{P:square2} we know that $\alpha_k$ and $\delta_k$ are surjective for $k \ge 3$.  Therefore it suffices to prove that  $\beta_k$ is surjective for all $k \ge 3$.

\medskip

Since $X$ is defined as the degeneracy locus of the matrix of the Petri map
\[
\mu_0(L): H^0(C,L)\otimes V \lra H^0(C,\omega_C)
\]
 a free resolution of $\I_{X/\bP}$ can be written under the form:
\[
\xymatrix{
0\ar[r]& H^0(C,L)\otimes \cO_{\bP}(-3) \ar[r]^-\psi & V^\vee\otimes\cO_{\bP}(-2)\ar[r]& \I_{X/\bP}\ar[r]&0}
\]
Restricting to $X$ we obtain:
\[
\xymatrix{
0\ar[r]&\R\ar[r]&H^0(C,L)\otimes \cO_X(-3H)\ar[r]^-\psi & V^\vee\otimes\cO_X(-2H)\ar[r]&
N^\vee_{X/\bP}\ar[r]&0}
\]
where neither $\R$ nor $N^\vee_{X/\bP}$ is locally free because the rank of $\psi$ drops further on $\bP W$.

Restricting to $C$, which is disjoint from $\bP W$, we obtain from the base-point-free pencil trick an exact sequence as follows:
\[
\xymatrix{
0\ar[r]&\omega_C^{-3}L^{-1}\ar[r]&H^0(C,L)\otimes\omega_C^{-3}\ar[r]& V^\vee\otimes\omega_C^{-2}\ar[r]& N^\vee_{X/\bP|C}\ar[r]&0}
\]
From  this sequence we deduce the exact sequence:
\[
\xymatrix{
0\ar[r]& \omega_C^{-3}L\ar[r]& V^\vee\otimes\omega_C^{-2}\ar[r]& N^\vee_{X/\bP|C}\ar[r]&0}
\]
Twisting this one by $\omega_C^3$ we obtain the exact sequence:
\[
\xymatrix{
0\ar[r]& L\ar[r]& V^\vee\otimes\omega_C\ar[r]& N^\vee_{X/\bP|C}\otimes\omega_C^3\ar@{=}[d]\ar[r]&0\\
&&&N^\vee_{X/\bP}(3H)\otimes \cO_C}
\]
Since the induced homomorphism:
\[
H^1(C,L) \lra V^\vee\otimes H^1(C,\omega_C)
\]
is the isomorphism given by Serre's duality, we obtain that the homomorphism
\[
\phi:V^\vee\otimes H^0(C,\omega_C) \lra  H^0(C,N^\vee_{X/\bP}(3H)\otimes \cO_C)
\]
is surjective. Now   the diagram:
\[
\xymatrix{
V^\vee\otimes H^0(X,H)\ar[d]^-\cong \ar[r]& H^0(X,N^\vee_{X/\bP}(3H))\ar[d]^-{\beta_3}\\
V^\vee\otimes H^0(C,\omega_C)\ar[r]^-\phi& H^0(C,N^\vee_{X/\bP}(3H)\otimes \cO_C)}
\]
shows that $\beta_3$ is surjective. The surjectivity of $\beta_k$ for $k \ge 4$ is proved similarly. 
\qed
\vskip 1 cm

\part{}

\vskip 0.8 cm
\section{Surfaces with canonical sections}\label{S:fake}

We follow the terminology and the notation of \cite{E84}, \cite{E83},\cite{CL02}.

\begin{defin} A surface with canonical sections is a projective surface $\ov S \subset \mathbb{P}^g$,
with $g\geq3$, whose general section is a smooth canonical curve of genus $g$. We also require that $\ov S$ is not a cone.

\end{defin}

From the definition it follows that $\ov S$ is irreducible with at most  isolated singularities. Consider the diagram

\be\label{E:fakediag}
\xymatrix{
S\ar[d]_p\ar[r]^q&\ov S\subset\bP^g\\
S_0
}
\ee

where $q$ is the minimal resolution and $S_0$ is a minimal model of $S$. The smooth surface  $S$ is equipped with a  big and nef  divisor $C$, where  $\cO_S(C)=q^*\cO_{\ov S}(1))$,
with
 $$
 C^2=2g-2\,, \qquad K_S \otimes \cO_C=\cO_C \,,\qquad \, h^0(S, C)=g+1\,.
 $$

In \cite{E83}, \cite{E84}  Epema proves the following result (see also \cite{Um81} and Proposition 2.2 in \cite{CL02}).

\begin{prop}[Epema \cite{E84}]\label{epema} Let  $\ov S \subset \bP^g$ be a surface with canonical sections. Then

a) $\ov S$ is projectively normal,

b) $S$ is either a  minimal $K3$ surface (possibly with rational singularities) or a ruled surface,

c) $h^0(S, -K_S)=1$.
\end{prop}

 If $\ov S$ is not a   K3 surface, we have two possibilities for $S_0$:
either

$$
S_0=\bP^2\,.
$$
or
$$
S_0=\bP E\overset \pi\longrightarrow \Gamma
$$
where $\Gamma$ is a smooth curve of genus $\gamma$ and $E$ is a  rank-2 vector bundle on $\Gamma$.

\subsection{Notation for the case $S_0=\bP^2$}\label{caso-p2}

From Proposition  \ref{epema},  there is a unique section  
$J \in H^0(S, -K_S)$. This curve is contracted to a point by   $q$. We set
$$
C_0=p(C)\subset \bP^2\,,\qquad J_0=p(J) \subset \bP^2
$$
Then $C_0$ is an irreducible plane curve of degree $d$ and $J_0$ is a plane cubic passing through the singular points of $C_0$.
We denote by the same symbol $\ell$, the class of a line in $\bP^2$ and its pull-back, via $p$, to $S$.
We have
 \be\label{1}
C\sim d\ell-\sum_{i=1}^h\nu_i E_i\,,\qquad \nu_1\geq \nu_2\geq\cdots\geq\nu_h\geq 1
\ee
where $E_i$ is a chain of smooth rational curves, as in Section \ref{S:digression}, with $E_i\cdot E_j=-\delta_{ij}$. We set
$$
p_i=p(E_i)\,\quad i=1,\dots,h
$$
where $p_1,\dots, p_h$ are  (possibly  infinitely near)  points in $\bP^2$.

We have
$$
K_S\sim -3\ell+\sum_{i=1}^h E_i\,,\qquad J\sim 3\ell-\sum_{i=1}^h E_i\,,\qquad K_S\cdot C= J\cdot C=0, \quad J^2=K_S^2=9-h
$$

$$
|C|=\langle|K_S+C|+J, C\rangle\,,\qquad \dim|C|=g\,,\qquad \phi_{|C|}: S\to \ov S\subset\bP^g
$$
(here $\langle A , B\rangle$ stands for the linear system spanned by $A$ and $B$).
Finally,
since $K_S\cdot C=0$, we have
\be\label{bez}
3d=\sum_{i=1}^h\nu_i
\ee
so that
\be\label{genus}
2g-2=d^2-\sum_{i=1}^h\nu_i^2
\ee
Notice that $J^2<0$ and $J \cdot C=0$ imply $h \ge 10$.

\subsection{Notation for the case $S_0=\bP(E)$}\label{caso-pe}
In Proposition 1.4, p. 34,  of \cite{E84}, Epema proves that, after replacing $S_0$ with another minimal model, if necessary, we can assume that
\be
E=\cO_\Gamma\oplus\cO_\Gamma(D)\,,\qquad \deg D=-e\,,\qquad 
\ee
and furthemore $e>0$ (cf. Remark \ref{e>1}).
We denote by $Z$ the unique member of $|-K_S|$. We have 
$$
Z\cdot C=0\,,\quad C^2=2g-2\,,
$$
$$
|C|=\langle|K_S+C|+Z, C\rangle\,,\qquad \dim|C|=g\,,\qquad \phi_{|C|}: S\to \ov S\subset\bP^g
$$

We have the following standard  situation:
$$
\Pic(S_0)=\sigma\bZ\oplus \pi^*\Pic(\Gamma)\,,\quad \sigma^2=-e\,, \quad \sigma\cdot \pi^*(x) =1\,,\quad x\in \Gamma
$$
Also
$$
K_{S_0}\sim -2\sigma+\pi^*(K_\Gamma+D)
$$
Let $C_0$ be the image of $C$ under the morphism $p:S\to S_0$. Then
$$
C_0\sim a(\sigma-\pi^*(D))
$$
for some $a \ge 1$,
so that
$$
C\sim p^*(C_{0})-\sum_{i=1}^h\nu_iE_i\\
=ap^*(\sigma-\pi^*D)-\sum_{i=1}^h\nu_iE_i\\
$$

where, as before, $E_i$ is a chain of smooth rational curves with $E_i\cdot E_j=-\delta_{ij}$,  and where we assume
$$
\nu_1\geq \nu_2\geq\cdots\geq\nu_h\geq 1
$$
We also set
$$
p_i=p(E_i)\,\quad i=1,\dots,h
$$
where $p_1,\dots, p_h$ are (possibly infinitely near)  points in $\bP(E)$.
We then have
$$
\aligned
K_{S}&\sim p^*(K_{S_0})+\sum_{i=1}^hE_i
=p^*(-2\sigma+\pi^*(K_\Gamma+D))+\sum_{i=1}^hE_i\\
Z&\sim p^*(2\sigma-\pi^*(K_\Gamma+D))-\sum_{i=1}^hE_i
\endaligned
$$
Notice that, since $p^*\sigma$ and $Z$ are effective divisors with $p^*\sigma\cdot Z<0$, we must have
\be\label{anti-z}
Z=p^*\sigma+J
\ee
where $J$ is an effective divisor with
\be\label{comp-z}
J\sim p^*(\sigma-\pi^*(K_\Gamma+D))-\sum_{i=1}^hE_i
\ee
Notice that in case $\gamma=1$
\be\label{int-sigma-j}
p^*(\sigma)^2=-e\,,\quad J^2=e-h\,,\quad p^*(\sigma)\cdot J=0.
\ee
In the case $\gamma=1$ we have more precisely $-D=\sum_{i=1}^e x_i$, $\sigma-\pi^*(K_\Gamma+D)=\sigma+\sum f_i$ where $f_i=\pi^{-1}(x_i)$.  All the members of the anticanonical system $|K_{S_0}|$ are reducible and of the form $\sigma+J_0$ with $J_0\in |\sigma+\sum f_i|$. Moreover $\dim(|\sigma+\sum_i f_i|)=e$, as it follows from the exact sequence on $S_0$:
\[
\xymatrix{
0\ar[r] &\cO(\sum f_i) \ar[r]& \cO(\sigma+\sum f_i)\ar[r] & \cO_\sigma \ar[r] & 0}
\]
and from $H^1(\cO(\sum f_i))=0$.  An analogous exact sequence shows that, for each effective divisor $\sum_{k=1}^{e-1} y_k$ on $\Gamma$ of degree $e-1$ we have $\dim(|\sigma+\sum \pi^{-1}(y_k)|)=e-2$.  Therefore the general member of  $|\sigma+\sum_i f_i|$ is irreducible and, if $J_0\in |\sigma+\sum_i f_i|$ contains a fibre of $\pi$ as a component, then it splits as $J_0=\sigma+ \sum \pi^{-1}(y_i)$, for some $\sum y_i \in |D|$.  Accordingly,  $J\subset S$ is either irreducible or splits as $p^*\sigma+\sum \varphi_i$ where the $\varphi_i$'s are the proper transforms of fibres of $\pi$.

Epema's Proposition 1.4  in \cite{E84} (see also Proposition 2.6 in \cite{CL02}), contains the following facts and some more: see for instance Chapter IV in \cite{E84} for the slight modifications
in case $ \gamma=1$.  

\begin{prop}[\cite{E84}]\label{epema2} Suppose $S_0=\bP(E)$. After replacing $S_0$ with another minimal model we can assume that
\begin{subequations}\label{grp}
\begin{align}
&E=\cO_\Gamma\oplus\cO_\Gamma(D)\,,\qquad \deg D=-e\label{1}\\
&\text{Either}\quad D\sim-K_\Gamma\,,\quad \text{or}\quad h \geq 1\quad \text{and}\quad e > 2\gamma-2 \label{2b}\\
&\nu_i \leq a-1\,,\quad i=1,\dots,h \label{3}\\
 &\sum_{i=1}^h\nu_i=a(e-2\gamma+2)\label{4d}\\
 &2g-2=a^2e-\sum_{i=1}^h\nu_i^2 \label{5e}\\
 &2a^2(\gamma-1)+a(e-2\gamma+2) \leq 2g-2 \leq a^2e-a(e-2\gamma+2)   \label{6}\\
  &(a-1)h\leq2\left[g-1-a^2(\gamma-1)\right]\,, \quad\text{for}\quad a\geq 2\label{7}\\
  &\text{None of the base points of}\quad |C_0| \quad\text{lie on}\quad \sigma\subset S_0. \label{8}
\end{align}
\end{subequations}
\end{prop}
By virtue of (\ref{8}) we will often use the same symbol $\sigma$ to denote the minimal section $\sigma\subset S_0$
and its pull-back $p^*\sigma\subset S$,  under the minimal model map $p: S\to S_0$.

\begin{rem} \label{e>1}Observe that if $g\geq2$ the genus formula (\ref{5e}) gives that $e\geq1$.  We also  have $a\geq2$, by \eqref{3}.
\end{rem}


\section{Smoothing}

Let $g \ge 3$ and consider again the situation of diagram \eqref{E:fakediag}:
\[
\xymatrix{
S\ar[d]_-p \ar[r]^-q & \overline S \subset \bP^g \\
S_0}
\]
We keep the same notations as in Section \ref{S:fake}.

 \begin{lemma}\label{L:smooth2}
 Let $T_{\overline S} = Hom(\Omega^1_{\overline S},\cO_{\overline S})$.  Then
\[
H^2(\overline S,T_{\overline S})=0
\]
\end{lemma}

\begin{proof}
We have $T_{\overline S}= q_*T_S$  (Proposition  \ref{epema}, and \cite{BW74}, Prop. 1.2) and therefore
\[
H^2(\overline S,T_{\overline S}) = H^2(S,T_S)
\]
Since $p$ is a sequence of blow-ups at nonsingular points we also have (\cite{Sern06}, Ex. 2.4.11(iii))
\[
H^2(S,T_S) = H^2(S_0,T_{S_0})=0
\]
\end{proof}

\medskip

Let
 $L=\cO_{\overline S}(1)$ and let $c(L)\in H^1(\overline S,\Omega^1_{\overline S})$ be its Chern-Atiyah class. It corresponds to an extension:
\be\label{E:EL1}
0 \to \Omega^1_{\overline S} \to \mathcal{Q}_L \to \cO_{\overline S} \to 0
\ee
Set $\mathcal{E}_L= \mathcal{Q}_L^\vee$ and consider the dual exact sequence:
\begin{equation}\label{E:EL}
0\to \cO_{\overline S} \to \mathcal{E}_L \to T_{\overline S} \to 0
\end{equation}

\begin{lemma}\label{L:EL}
$H^2(\overline S,\mathcal{E}_L)=0$.
\end{lemma}

\begin{proof}  \ \ In all cases $\overline S$ has Gorenstein normal singularities  since $\overline S$ is projectively normal, $\omega_S=\cO_S(-Z)$ and $Z$ is   supported on the exceptional locus of $q$.   Moreover  
\[
\omega_{\overline S}=\cO_{\overline S}
\]
 by adjunction. 
Since by Serre duality
\[
H^2(\overline S,\cO_{\overline S})=H^2(\overline S,\omega_{\overline S})\cong H^0(\overline S,\cO_{\overline S})^\vee\cong\bC
\]
 it suffices to prove that the coboundary map:
\[
\partial:H^1(\overline S,T_{\overline S}) \to H^2(\overline S,\cO_{\overline S})
\]
from the exact sequence \eqref{E:EL}
 is non-zero. The map $\partial$ is the cup product with the extension class defining \eqref{E:EL}. This class is a non-zero element of $\Ext^1(T_{\overline S},\cO_{\overline S})$ and this space is the Serre dual of $H^1(\overline S,T_{\overline S})$. Hence $\partial$ is non-zero. 
\end{proof}

Consider the deformation functor $\mathrm{Def}_{(\overline S,L)}$ of (not necessarily locally trivial) deformations of the pair $(\overline S,L)$.  In \cite{AC10} it is shown that
\[
\mathrm{Def}_{(\overline S,L)}(\bC[\epsilon])= \mathrm{Ext}^1(\mathcal{Q}_L,\cO_{\overline S})
\]
and the obstruction space is contained in $\mathrm{Ext}^2(\mathcal{Q}_L,\cO_{\overline S})$.

 We have a morphism of functors:
\[\mu:\mathrm{Def}_{(\overline S,L)} \longrightarrow \coprod_{\bar s\in \mathrm{Sing}(\overline S)}  \mathrm{Def}_{\cO_{\overline S,\bar s}}
\]
defined by restriction. 

\begin{thm}\label{T:smooth}
The morphism $\mu$ is smooth. In particular, if $\overline S$ has smoothable singularities then   $[\overline S]\in \mathrm{Hilb}^{\bP^g}$ belongs to the closure of the family of nonsingular $K3$ surfaces of degree $2g-2$.
\end{thm}

\begin{proof}
The exact sequence coming from the local-to-global spectral sequence for Ext  implies the following  exact sequence (see \cite{AC10}, Thm. 3.1):
\[
\xymatrix{
0\to H^1(\overline S,\mathcal{E}_L )\to \mathrm{Ext}^1(\mathcal{Q}_L,\cO_{\overline S})
\ar[r]^-{d\mu}& H^0(\overline S,T^1_{\overline S}) \to H^2(\overline S,\mathcal{E}_L )}
\]
where $d\mu$ is the differential of $\mu$. From Lemma \ref{L:EL} we obtain that $d\mu$ is surjective.

Lemma \ref{L:EL} and the fact that $\ext^1(\mathcal{Q}_L,\cO_{\overline S})$ has finite support imply that
$$
\mathrm{Ext}^2(\mathcal{Q}_L, \cO_{\ov S})= H^0(\ext^2(\mathcal{Q}_L, \cO_{\ov S}))
$$ 
Moreover, since $\ov S$ is normal, we have (\cite{Sern06}, Prop. 3.1.14)
\[
\ext^2(\Omega^1_{\ov S}, \cO_{\ov S}) \cong T^2_{\ov S}
\]
From the exact sequence \eqref{E:EL1} we deduce that 
\[
\ext^2(\mathcal{Q}_L, \cO_{\ov S}) \subset \ext^2(\Omega^1_{\ov S}, \cO_{\ov S})
\]
Therefore the obstruction map:
\[
o(\mu): H^0(\ext^2(\mathcal{Q}_L, \cO_{\ov S})) \to H^0(\ov S, T^2_{\ov S})
\]
is injective. 
Therefore $\mathrm{Def}_{(\overline S,L)}$ is less obstructed than 
$\coprod_{\bar s\in \mathrm{Sing}(\overline S)}  \mathrm{Def}_{\cO_{\overline S,\bar s}}
$. 
and
this implies that $\mu$ is smooth.

If the singularities of $\overline S$ are smoothable from the smoothness of $\mu$ it follows that the pair $(\overline S,L)$ deforms to a nonsingular pair $(X,\cO_X(1))$.  Moreover, the fact that $H^1(L)=0$ implies that all the sections of $L$ extend and the theorem follows.
\end{proof}



%

\begin{cor}\label{C:smooth1}
With the same notation as before,  assume that we are in one of the    following cases:
\begin{itemize}

\item[(a)]  $S_0=\bP^2$, $K_S^2=9-h$. 

\item[(b)]  $S_0= \bP(E)$, where $E=\cO_\Gamma\oplus \cO_\Gamma(D)$, with $\Gamma$ a smooth curve of genus one and $\deg(D)=-e$, $K_S^2=-h$, $h\geq e$.
\end{itemize}

Then $\overline S\subset \bP^g$ deforms to a surface $X \subset \bP^g$ such that:

\begin{itemize}
\item Case (a): $X$ has a simple elliptic singularity of degree $h-9$, and its minimal resolution is the blow-up of $h$ distinct points on a nonsingular cubic in $\bP^2$.

\item  Case (b): $X$ has two simple elliptic singularities of degrees $e$ and $h-e$ respectively.
\end{itemize}

\end{cor}

\begin{proof}    
 Case (a). \ \  From Theorem 1 of \cite{Um81} 
it follows that
the singularity $\bar s\in \overline S$ satisfies $p_g(\bar s)=1$ and therefore it is  a minimally elliptic  singularity (\cite{Lau77}, Thm. 3.10).   This singularity, being obtained by blowing up $h$ (possibly infinitely near) points on a plane cubic $J_0$, can be deformed to a simple elliptic singularity obtained by making $J_0$ nonsingular and the $h$ points distinct.  This deformation can be obtained by deforming $\overline S$ in $\bP^g$, by Theorem \ref{T:smooth}.

Case (b). \    Theorem 1 of \cite{Um81} implies that $\overline S$ has either two simple elliptic  singularities or one normal Gorenstein singularity of geometric genus two. In the first case there is nothing to prove.  
The second case arises when the component $J$ of $Z$ (cf. (\ref{anti-z}), (\ref{comp-z})) splits as  $p^*\sigma+\sum \varphi_i$ where the $\varphi_i$'s are the proper transforms of fibres of $\pi$.
Since $h \ge e$  we can   move the points $p_1, \dots, p_h$  so that they do not belong to a reducible curve of the form  $\sum \pi^{-1}(y_i)$, for some $\sum y_i \in |D|$ and thus making $J$ irreducible. This means that the second case is   a degeneration of the first one obtained by moving the points $p_1,\dots, p_h\in S_0$ so as to make $J$ reducible. Theorem \ref{T:smooth} implies that this deformation can be realized by deforming $\overline S$ in $\bP^g$.

\end{proof}

\begin{cor}\label{C:smooth2}
In the same situation of Corollary \ref{C:smooth1}, assume further the following:
\begin{itemize}
\item Case (a):  $-9 \le K_S^2 \le -1$.

\item Case (b):  $1\le e \le h-1$ and   $2 \le h \le 10$.
\end{itemize}
Then $\overline S$ deforms in $\bP^g$  to a nonsingular K3 surface  of degree $2g-2$. 
\end{cor}

\begin{proof}
Corollary \ref{C:smooth1} implies that $\overline S$ deforms to a surface  $X\subset \bP^g$ having only simple elliptic singularities. Such singularities are analytically equivalent to a cone $C_E$ over an elliptic curve $E\subset \bP^{k-1}$ embedded with degree $k$, if $k \geq 3$, or  to a hypersurface singularity  if $k=1,2$ (\cite{Lau77}, Theorem 3.13). In the cases considered   $k=-K_S^2$ in case a), or $k=e$ and $k=h-e$ in case b). 
In all such cases (and in the cases $k=1, 2$)  the singularities of $X$  are smoothable (see \cite{Pink74}). Therefore the Corollary follows by applying Theorem \ref{T:smooth} to $X$.

\end{proof}

\begin{rem}\rm   Theorem \ref{T:smooth} is related to results of Karras, Looijenga, Pinkham, Wahl  \cite{Ka77,Loo81,Wah81, Pink97} about the possibility of inducing all deformations of a normal (elliptic) singularity of a compact surface by means of global {\it abstract} deformations of the surface. Our result involves {\it a polarization}  in its statement, assuring that, in the case of surfaces with canonical sections in $\bP^g$, the deformations of the singularity can be realized by   deforming the surface in $\bP^g$.  At the first order level one can compare the two cases (abstract vs embedded deformations) by means of the following diagram: 
\[
\xymatrix{
0 \ar[r]& H^1(\mathcal{E}_L)\ar[d]\ar[r]&\mathrm{Ext}^1(\mathcal{Q}_L,\cO_{\overline S})\ar[d]\ar[r]&H^0(T^1_{\overline S})\ar@{=}[d]\ar[r]&H^2(\mathcal{E}_L)=0\\
0\ar[r]&H^1(T_{\overline S})\ar[d]^-\partial\ar[r]& \mathrm{Ext}^1(\Omega^1_{\overline S},\cO_{\overline S})\ar[d]\ar[r] &H^0(T^1_{\overline S})\ar[r]& H^2(T_{\overline S})=0 \\
& H^2(\cO_{\overline S})\ar[d]\ar@{=}[r]&H^2(\cO_{\overline S})\ar[d] \\
&0&0}
\]
where the horizontal sequences are the local-global sequences of the Ext functors.
\end{rem}

More recently smoothing techniques similar to ours have been applied to the construction of surfaces of general type (see \cite{KLP12} and references therein).

\section{The main theorem for BNP curves on surfaces with canonical sections}\label{main-fake}

Our  main theorem for BNP curves on surfaces with canonical sections is the following.

  \begin{thm}\label{Main11} Let $C$ be a Brill-Noether-Petri   curve of genus $g\geq 12$. Then $C$ lies on a polarized K3 surface or on a limit of such  if and only if the Wahl map is not surjective.
  \end{thm}
  
 Let us go back to diagram (\ref{E:fakediag}) and consider the curve $C_0\subset S_0$.
In the next three sections we will  analyse  the restrictions imposed on the singularities of  $C_0$
when we assume that $C$ is a BNP-curve. We first study the case in which $S_0=\bP(E)\to\Gamma$, is a ruled surface with $g(\Gamma)=\gamma$, and in  Sections \ref{genus>1}, \ref{genus-1}, \ref{genus-0},  we will address the cases in which $\gamma>1$, $\gamma=1$ and $\gamma=0$, respectively.
In section \ref{p2} we will  study the case in which $S_0=\bP^2$.
\vskip 0.2 cm 
Referring to the notation we  established in Sections  \ref{caso-p2} and  \ref{caso-pe},
we will prove the following proposition.

\begin{prop}\label{main-can-sect} Let $\ov S\subset \bP^g$ be a surface with canonical sections,
and let $C$ be a general hyperplane section of $\ov S$.
Suppose that $C$ is a BNP curve.
\vskip 0.2 cm 
A) Suppose $S_0=\bP(E)$, with $E$ a rank-two vector bundle
over a smooth curve $\Gamma$ of genus $\gamma$. Assume that $g\geq 12$,
then $\gamma\leq1$.

\vskip 0.1 cm 
B) Suppose $S_0=\bP(E)$, with $E$ a rank two vector bundle
over an elliptic curve   $\Gamma$. 
Assume that $g\geq 12$. Then $e+1 \le h \le 7$. (recall that $\sigma^2=-e$ and $J^2=e-h$).   

\vskip 0.1 cm 

C) The case $S_0=\bP(E)$, with $E$ a rank two vector bundle
over a smooth rational curve   $\Gamma$,  can be reduced to next one.
\vskip 0.1 cm 

D) Suppose $S_0=\bP^2$, and assume that  $g\geq 12$. Then $10 \le h \le 18$ and therefore $K_S^2=9-h\ge -9$.

\end{prop}
Let us assume Proposition \ref{main-can-sect} and let us prove the main theorem.

\begin{proof} (of Theorem \ref{Main11}). Because of Wahl's Theorem  \cite{Wah87} (see also Beauville  and Merindol's proof 
\cite{BM87}), there is only one implication to be proved. Let then  $C$ be a Brill-Noether-Petri  curve of genus $g\geq 12$, for which the  Wahl map is not surjective. From Theorem \ref{T:main},
and from Wahl's Theorem \ref{wahl1} we conclude that $C$ is extendable.
Let then $\ov S\subset \bP^g$ be a surface with canonical sections, having $C$ as one of its hyperplane sections. If $\ov S$ is not already a K3 surface,  we are  in the situation of diagram (\ref{E:fakediag}).
By Corollary \ref{C:smooth1} we may assume that the anti-canonical divisor of the surface $S$ is smooth (with one or two connected components). By Proposition \ref{main-can-sect},  we are either in case B) or in case D) and, moreover,
 each component of the anti-canonical divisor of $S$
has self-intersection $\ge -9$. It now suffices to apply Corollary \ref{C:smooth2} to deduce that $\ov S$ is a limit of K3 surfaces.

\end{proof}

\vskip 0.8 cm

\subsection{Ruled surfaces of genus $\gamma>1$}\label{genus>1}
In this section the surfaces $S$ and   $S_0=\bP(E)$ and the curve  $C$ will be as described in the previous Subsection \ref{caso-pe}. 
The next Proposition gives  item A) of Proposition \ref{main-can-sect}.
\begin{prop}\label{gamma>1}  
Assume that $C$ is a BNP-curve of genus $g\geq 12$. Then $\gamma\leq1$ and $a\geq 3$.
\end{prop}

 \begin{lemma} \label{BNgen} Suppose that $C$ is a BNP-curve.
Then
\be\label{pencil-a}
a(\gamma+3)\geq g+2
\ee
\end{lemma}
\begin{proof} Consider a 
$g^1_k$ on  $\Gamma$ and lift it to  $C$ via the composition
 $\pi \cdot  p : C \to \Gamma$. This is a degree
$a$ morphism. We then get a pencil of degree $ak$ on $C$.
As $C$ is BNP
we must have
$$
2ak \geq g+2
$$
If the $g^1_k$ on  $\Gamma$ is of minimal degree we have 
$$
2 k \leq \gamma+3
$$
so that $a(\gamma+3) \geq 2ak \geq g+2$.

\end{proof}

\begin{lemma} \label{ineq}
If   $C$ is a BNP curve we have $(a-2)(\gamma -1)+\frac{e}{2} \leq 4-\frac{3}{a}$.
\end{lemma}
\begin{proof}
The first inequality in (\ref{6}) can be written as 
$$
 1+ (a^2-a)(\gamma-1)+\frac{ae}{2} \leq g
 $$
 This, together with (\ref{pencil-a}) gives
$$
 1+ (a^2-a)(\gamma-1)+\frac{ae}{2} \leq g \leq a(\gamma+3)-2=a(\gamma-1)+4a-2
 $$
and therefore
$$
  (a^2-2a)(\gamma-1)+\frac{ae}{2} \leq 4a-3
 $$
Dividing by $a$ we get the Lemma.
\end{proof}

\begin{proof} (of Proposition \ref{gamma>1}).
Assume that $a=2$. 
From   Lemma \ref{ineq} we get $\frac{1}{2}e \leq 4-\frac{3}{a}=\frac{5}{2}$ so that $e \leq 5$.
From (\ref{5e}) we then get $g \leq 1+2e \leq 11$, which is excluded by hypothesis.
We may then assume  $a \geq 3$. 
  Since $2\gamma-2 \leq e$,
Lemma \ref{ineq} gives
$$
2(\gamma-1)\leq (a-1)(\gamma-1)<4
$$
so that $\gamma \leq 2$.
Suppose  $\gamma=2$, then $a \leq 4$. If $a=4$, 
from Lemma \ref{ineq} we get $e \leq 2$ so that $e=2$.
From  (\ref{4d}) we deduce that $C_0$ is smooth and  $g=17$. But then,
since $\Gamma$ is hyperelliptic, there is a $g^1_8$ on  $C$  with negative Brill-Noether number
contrary to the assumption.
In conclusion, if  $\gamma=2$, we must have $a=3$. Pulling back on $C$ the $g^1_2$ on $\Gamma$
the BNP property yields
$2a\geq(g+2)/2$, giving
 $g \leq 10$.
\end{proof}

\vskip 0.8 cm

\subsection{Elliptic ruled surfaces}\label{genus-1} In this section our aim is to prove  item B) of Proposition \ref{main-can-sect}.
When $\gamma=1$, the setting is the following. The minimal model $S_0$ of $S$
is a ruled surface 
$$
\pi: S_0=\bP E\to \Gamma
$$
over an elliptic curve $\Gamma$ and
 $E=\cO_\Gamma\oplus\cO_\Gamma(D)$. Also

$$
K_{S_0}\sim -2\sigma+\pi^*D\,,\quad \deg D=-e, \quad K_{S_0}^2=0
$$
and
$$
K_{S}\sim p^*(-2\sigma+\pi^*D)+\sum_{i=1}^hE_i\,, \quad K_S^2 = -h
$$
where  $\sigma \subset S_0$ is a minimal section  and $p: S\to S_0$ is the minimal model map.
The anticanonical divisor $Z\in|-K_S|$ is given by 
\be\label{z-and-j}
Z\sim p^*\sigma+ J\,,\quad\text{where}\quad J\sim p^*(\sigma-\pi^*D)-\sum_{i=1}^h E_i
\ee
and $p^*\sigma$ and $J$ are effective   divisors such that:
\[
(p^*\sigma)^2=\sigma^2=-e, \quad J^2=e-h, \quad p^*\sigma\cdot J=0
\] 
The curve $C\subset S$ is given by
$$
C\sim ap^*(\sigma-\pi^*D)-\sum_{i=1}^h\nu_iE_i\,,\qquad \nu_1\geq\dots\geq\nu_h\geq 1, \quad a \ge 3\,,\qquad e\geq 1
$$
The following Proposition gives  item B) of Proposition \ref{main-can-sect}.

\begin{prop}\label{caso-ellittico}Suppose that $C$ is a BNP curve of genus $g\geq 12$. Then 
$e+1 \le h \le 7$. 
\end{prop}

The probe we will often use to test the BNP property of $C$ is the pencil on $S$ given by

\be\label{pencilM}
|M|=|p^*(\sigma-\pi^*D)-E_1-\cdots-E_{e-1}|
\ee
The fact that, indeed,  this is at least a pencil follows from the fact that
$$ 
\aligned
h^0(S, p^*(\sigma-\pi^*D))
&=h^0(S_0, \sigma-\pi^*D)
=h^0(\Gamma, \pi_*(\sigma-\pi^*D))\\
&=h^0(\Gamma, (\cO_\Gamma\oplus\cO_\Gamma(D))\otimes\cO_\Gamma(-D))=e+1\\
\endaligned
$$
and from (\ref{2b}). Set $M_C=M\otimes\cO_C$. Consider  the exact sequence
\be\label{probe0}
0\to \cO_S(M-C)\to \cO_S(M)\to \cO_C(M_{C})\to 0
\ee
By restricting to $\pi^*(x)$ we see that 
$h^0(S, M-C)=0$, so that the  linear series $|M_{C}|$ is, at least, a pencil.
A way to enforce the BNP property of $C$ will be to insist that 
\be\label{probe1}
h^0(C, KM_{C}^{-2})=0
\ee
Looking at the exact sequence
$$
0\to \cO_S(-2M)\to \cO_S(C-2M)\to \cO_C(KM_{C}^{-2})\to 0
$$
and noticing that  $h^0(S, -2M)=0$, we see that a necessary condition for (\ref{probe1}) to hold is that
\be\label{probe2}
h^0(S, C-2M)=0
\ee
holds. Hence if $C$ is a BNP curve, (\ref{probe2}) must hold.
We come to the first Lemma.

\begin{lemma}\label{fund}
Suppose that  $C$ is a  BNP curve and  let  $g>4$. Then

i) $4a\geq g+4$

ii) $ \nu_e=a-1$
\end{lemma}
\proof

The elliptic curve  $\Gamma$ parametrizes infinitely many   $g^1_2$'s on itself. As a consequence, $W^1_{2a}(C)$
contains an elliptic curve. Since $W^1_{2a}(C)$ is connected and $C$ is  BNP curve, by the genus formula for
$W^1_d$ when $\rho=1$, (see e.g. \cite{Ort13}, p. 811)
we must have  
$$
2 \le \dim W^1_{2a}(C)=\rho(g,2a,1)=4a-g-2
$$
proving i). Let us prove ii). Suppose it does not hold.
From (\ref{3}) we then get  $\nu_e\leq a-2$. As a consequence we have:
$$
\aligned 
&\nu_1,\dots,\nu_{e-1}\leq a-1\,,\\
&\nu_j\leq a-2\,,\quad j=e,\dots, h
\endaligned
$$
This means that
\be\label{box}
\aligned 
(a-2)-(\nu_i-1)&\geq 0\,,\quad i\leq e-1\\
(a-2)-\nu_j&\geq 0\,,\quad j\geq e\\
\endaligned
\ee
Now consider the pencil $|M|$ defined in \eqref{pencilM}. Using \eqref{box} we get
$$
\aligned
C-2M&=(a-2)p^*(\sigma-\pi^*D)-\sum_{i=1}^{e-1}(\nu_i-2)E_i-\sum_{i=e}^{h}\nu_iE_i\\
&=(a-2)J+\sum_{i=1}^{e-1}[(a-2)-(\nu_i-2)]E_i+\sum_{i=e}^{h}[(a-2)-\nu_iE_i]
\endaligned
$$
where $J$ is defined in (\ref{z-and-j}). But then $C-2M$ is an effective divisor
contrary to  assumption \eqref{probe2}.
\endproof

The preceding Lemma, (\ref{3}), (\ref{4d}) and (\ref{5e}), put us in the following situation

\be
\label{situation}
\aligned
4a&\geq g+4\\
\nu_1&=\cdots=\nu_{e}=a-1\geq\nu_{e+1}\geq\cdots\geq \nu_h\geq1\\
2g-2&=e(2a-1)-\sum_{i=e+1}^h\nu_i^2\\
e&=\sum_{i=e+1}^h\nu_i
\endaligned
\ee

\begin{proof} (of Proposition \ref{caso-ellittico}).  From (\ref{7}) and the first inequality in (\ref{situation}), we get $h\leq 7$. In particular $e+1\leq h\leq 7$.  
\end {proof}

\vskip 0.8 cm 
\subsection{Rational ruled surfaces}\label{genus-0}

In this section we will assume that $\gamma=0$. 
Our aim is to prove item C) of Proposition  \ref{main-can-sect}.
The setting is the following:
the minimal model of  $S$ is the rational ruled surface $ S_0=\mathbb F_e$, i.e.
$$
\pi: S_0=\mathbb F_e=\bP(E)\to \Gamma
$$
where $E=\cO_\Gamma\oplus\cO_\Gamma(-e)$.
We recall that
$$
K_{\mathbb  F_e}\sim -2\sigma-(e+2)f
$$
where $f$ is the class of a fiber 
and that
$$
K_{S}\sim p^*(-2\sigma-(e+2)f)+\sum E_i
$$
where $\sigma \subset \mathbb  F_e$ is a minimal section with $\sigma^2=-e$.
From Proposition \ref{epema2}, there exists a unique divisor
$$
Z \in H^0(S, -K_S)
$$
such that $Z \cdot C=0$.
Furthermore the class of $C \subset S$ is 
$$
C=p^*(a\sigma+aef)-\sum_{i=1}^h\nu_iE_i
$$ 

It is convenient to write $(S_e, C_e, Z_e, p_e, S_{0,e}, C_{0,e}, Z_{0,e}, \sigma_e)$ instead of $(S, C, Z, p, S_0, C_0, Z_0, \sigma)$. Choose a point 
 $$
 x_e \in Z_e, \qquad x_e \notin C_e \cup \sigma_e
 $$
Let us  perform the elementary transformation of $S_e$ centered at $x_e$ (see \cite{Dol12}, section 7.4): we first  blow-up $S_e$ at $x_e$:
  $$
  s_e:X_e \longrightarrow S_e
  $$
with exceptional divisor $E_{x_e} \subset X_e$. 
Since the strict transform $F_{x_e}=s^*f_{x_e}-E_{x_e}$
is an exceptional curve of the first kind we can contract it.
The blow down of $F_{x_e}$ is a morphism 
$$
t_{e-1}: X_e \longrightarrow S_{e-1}
$$ 
where $S_{e-1}$ is a non-minimal  ruled surface, with a section $\sigma_{e-1}:=t_{e-1}(s_e^{-1})(\sigma_e)$
of self-intersection $$
\sigma_{e-1}^2=(t_{e-1}^*(\sigma_{e-1}))^2=(s_e^*(\sigma_e)+F_{x_e})^2=-e+1
$$
By construction, it follows  that, if
$$
p_{e-1}: S_{e-1} \longrightarrow S_{0,e-1}
$$
is the blow down of $E_1, \ldots, E_h$, then $S_{0,e-1}$ is the minimal
ruled surface $\mathbb F_{e-1}$.
Summarising the construction, we have a diagram:

\be
\xymatrix
{&X_e\ar[dl]_{t_{e-1}}\ar[dr]^{s_e}\\
S_{e-1}\ar[d]_{p_{e-1}}&&S_e\ar[d]^{p_e}\\
\mathbb  F_{e-1}&&\mathbb  F_e
}
\ee
in which $F_{x_e}$ is $(t_{e-1})$-exceptional and $E_{x_e}$ is $(s_e)$-exceptional.
If, by a slight abuse of notation, we denote by $f$ a general fiber in both $S_e$
and  $S_{e-1}$, we have on $X_e$ the linear equivalence 
$$
s_e^*f=t_{e-1}^*f \sim E_{x_e}+F_{x_e}\,.
$$
Let us denote by $\ov D$ the strict transform of a divisor $D$ under $s_e$ or $t_{e-1}$.
By construction, we  have:

$$ 
s_e^*C_e= \ov C_e, \qquad s_e^*\sigma_e= \ov \sigma_e=t_{e-1}^* \sigma_{e-1} -F_{x_e}  
$$

Consider the image under $t_{e-1}$ of the strict transform of $C_e$ and set:
 $$
C_{e-1}:=t_{e-1}(\ov C_e)
$$
This is an irreducible curve birational to $C_e$, which has
a point of multiplicity $a$ at the point  $t_{e-1}(F_{x_e})=C_{e-1} \cap \sigma_{e-1}$.
Let us also consider the projection of $C_{e-1}$ to $\mathbb F_{e-1}$
 $$
C_{0, e-1}:=p_{e-1}( C_{e-1})
$$
The class of $C_{0, e-1} \subset \mathbb F_{e-1}$ is $a\sigma_{e-1}+eaf$, because
 $t_{e-1}^*(p_{e-1}^*(\sigma_{e-1}))=s_e^*(p_e^*(\sigma_e)))+F_{x_e}$ and, by construction,
$ t_{e-1}^*(C_{e-1})=s_e^*(C_e)+aF_{x_e}$.
We can then write
$$
 \ov C_{e-1}= t_{e-1}^*(p_{e-1}^*(a\sigma_{e-1}+eaf))-aF_{x_e}-\sum_{i=1}^h\nu_iE_i
 $$

Finally, the canonical formula on $X_e$ gives:
$$
K_{X_e}=s_e^*(K_{S_e})+E_{x_e}=t_{e-1}^*(K_{X_{e-1}})+F_{x_e}
$$
so that $\ov Z_e \in H^0(X_e, s_e^*(-K_{S_e}-E_{x_e})=H^0(X_e, -K_{X_e})$
is also the strict transform of a unique divisor
$$
Z_{e-1} \in H^0(S_{e-1}, -K_{S_{e-1}})
$$
whose projection $$Z_{0,e-1}:=p_{e-1}(Z_{e-1}) \subset \mathbb F_{e-1}$$
contains all the singular locus of $C_{0, e-1}$.
The following Lemma gives a proof of item C) of Proposition \ref{main-can-sect}.

\begin{lemma}\label{ratruled}
If $S_0=\bP (E)$ there exists another minimal model $p': S \to \bP^2$.
\end{lemma}
\proof

We start from the usual diagram:

$$
\xymatrix{S\ar[d]_p\ar[r]^q& \ov S\\
S_0=\bP(E)
}
$$

Applying repeatedly the previous construction, we get a diagram
$$
\xymatrix{X\ar[d]_{p_1'}\ar[r]^{s}&S\ar[d]_{p}\ar[r]^{q}&\ov S\\
\bP^2 & S_0
}
$$
where $s$ is the blow up of $S$ at points $x_2, \ldots, x_e$, with exceptional divisors $E_{x_2}, \ldots, E_{x_e}$,
 and $p_1'$ is the composition of  the following maps: the blow-up of $\bP^2$ at a point $P$,
the blow-up $\wt{\mathbb F}_1$ of $\mathbb F_1$ with exceptional divisors $E_1, \ldots, E_h$,  the blow-up $X$  of  $\wt{\mathbb F}_1$ 
 with exceptional divisors $F_{x_2}, \ldots, F_{x_e}$.
The proper transform of $C \subset S$ on $X$ maps via $p_1'$, onto a plane curve $C_1$
whose singular locus is contained in a unique anticanonical divisor.
In order to prove the Lemma, we need to show that there exists a diagram:

$$
\xymatrix{X\ar[d]_{p_1'}\ar[r]^{p_1''}& \bP^2\\
\bP^2
}
$$

where $X$ dominates the graph of the birational map $p_1'' \cdot p_1'^{-1}$, the divisors $E_{x_2}, \ldots, E_{x_e}$
are $p_1''$-exceptionals and the points $p_1''(E_{x_i})$ for $i=2, \ldots, e$ do not lie on the image $C_0 \subset \bP^2$ of the proper transform of $C_1$.
Indeed, if we have such a diagram, the map  $p_1''$ factors through $s$ because the divisors $E_{x_2}, \ldots, E_{x_e}$
are also $s-$exceptional and we obtain 
a diagram

$$
\xymatrix{S\ar[d]_{p'}\ar[r]^q& \ov S\\
\bP^2
}
$$

where $p_1''=p' \cdot s$.
In order to construct $p_1''$, let us consider the plane curve $C_1$.
The plane curve $C_1$ has degree $ae$, it has a point $P$ of multiplicity $a(e-1)$, and 
$e-1$ points $f_{x_2}, \ldots, f_{x_e}$  of multiplicity $a$, infinitely near to $P$.
  Furthermore it has points $y_1, \ldots, y_h$, distinct from $P$ and not infinitely near to $P$, with multiplicities
$a >\nu_1, \ldots, \nu_h \geq 1$. Note that the proper transform under $p_1'$ of the lines $l_{x_i}$
through $P$ in the directions $f_{x_i}$ are the exceptional curves $E_{x_i}$ in $X$ and note also that for each $i=2, \ldots, e$
we have $C \cdot E_{x_i}=0$.
If we  perform a  first (degenerate) standard quadratic transformation $\Phi_2$ of 
the plane (see \cite{Dol12} example 7.1.9) with centers at $P$, at $f_{x_2}$ and at $y_1$ (notice that $h >0$),
we obtain a plane curve $C_1'$ of degree  $ea-\nu_1$. The curve $C_1'$ has the same singular points of $C_1$ except
for the three chosen points, which are replaced by points 
of multiplicities $(0, (e-1)a-\nu_1, a-\nu_1)$. The point $x_2:=\Phi_2(E_{x_2})$ has, by construction, multiplicity zero
and does not to lie on $C_1'$.
The points $f_{x_3}, \ldots, f_{x_e}$ will be infinitely near to the point of multiplicity $(e-1)a-\nu_1$ and we can iterate
such a construction, because we observed that for each $i=2, \ldots, e$ we have $C \cdot E_{x_i}=0$.
We obtain in such a way the contraction of the divisors $E_{x_i}$ with the property that the points $\Phi_i(E_{x_i})$
do not belong to the image $C_0$ of the curve $C_1$.

This means that the composition $\Phi_e \cdot \cdots \cdot \Phi_2$ is the desired birational transformation inducing the sought for
map $p_1'':X \longrightarrow \bP^2$.

\qed

\vskip 0.8 cm

\subsection{The case of $\bP^2$}\label{p2}

We now consider the case in which the minimal model of $\ov S$ is $\bP^2$.
Our aim is to prove item D) of Proposition  \ref{main-can-sect}.
As usual we consider the desingularization 
$$
q:S\longrightarrow \ov S\subset\bP^g
$$
and the natural morphism
$$
p:S\longrightarrow S_0=\bP^2
$$
From  Proposition  \ref{epema},  there is a unique section $J \in H^0(S, -K_S)$
which is contracted to a point by $q$.
Set $C_0=p(C)\subset \bP^2$ and  $J_0=\pi(J) \subset \bP^2$.
The plane curve $J_0$ is a cubic passing through the singular points of $C_0$.
We denote by the same symbol $\ell$ the class of a line in $\bP^2$ and its pull-back on $S$.
 We have

$$
C\sim d\ell-\sum_{i=1}^h\nu_i E_i\,,
$$
\be\label {E:nui1}
 \nu_1 \geq \ldots \geq \nu_h \geq 1
\ee

$$
K_S\sim -3\ell+\sum_{i=1}^h E_i
$$
The condition $J\cdot C=0$ gives
\be\label{{E:3d}}
3d=\sum_{i=1}^h\nu_i
\ee
Furthermore, since $|C|$ contracts $J$, we have $J^2 <0$ so that, in particular $h \ge 10$.
By Corollary \ref{C:smooth1}, we may assume that $J_0$ is smooth, and that
$P_1=p(E_1),\dots,P_h=p(E_h)$ are distinct points of $J_0$, no three on a line.
Assume that  $\nu_1+\nu_2+\nu_3 > d$.  The quadratic transformation of $\bP^2$ centered at $P_1,P_2,P_3$ replaces $C_0$ by $\wt C_0$ of degree $\wt d = 2d-(\nu_1+\nu_2+\nu_3)< d$ and meeting the smooth cubic $\wt J_0$, the
proper transform of $J_0$, again at $h$ points $\wt P_1, \dots, \wt P_h$, of multiplicities say $\wt\nu_1 \ge \dots \ge \wt\nu_h$.  By Corollary \ref{C:smooth1} again, we may assume that 
$\wt P_1, \dots, \wt P_h$ are distinct and no three on a line. By repeating this process we can arrive to a situation where 
\be\label{assump1}
\nu_1+\nu_2+\nu_3\leq d
\ee
and therefore we may assume that this condition holds.
From now on we will set
$$
d=3n+\epsilon, \qquad \epsilon=0,1,2
$$
We can then assume that
\be\label{assump2}
\nu_3\leq n
\ee
We find it convenient to introduce  the following notation:
\[
b_i =n- \nu_i\,,\quad i=1,\dots,h\,.
\]
We can then write the class of $C$ in the following way:
\[
C = -nK_S+\epsilon\ell+\sum_i b_iE_i = (3n+\epsilon)\ell-\sum_1^h (n-b_i)E_i
\]

The inequalities \eqref{E:nui1}   and \eqref{assump1} imply that:
\begin{equation}\label{E:bi1}
b_1\le \cdots \le b_h
\end{equation}
and that 
\begin{equation}\label{E:bi2}
b_1+b_2+b_3+\epsilon \ge 0
\end{equation}
In particular $0\le b_3$.   We also have:
\[
K_S\cdot C = 0, \quad C^2 = d^2-\sum \nu_i^2 = (3n+\epsilon)^2-\sum_i(n-b_i)^2
\]
and therefore:
\begin{align}\label{E:degM}
g:=g(C)&= \frac{1}{2}(K_S+C\cdot C)+1=\frac{1}{2}C^2+1 = \notag \\
&=\frac{1}{2}[d^2-\sum \nu_i^2]+1 \notag \\
&= \frac{1}{2}\left[(3n+\epsilon)^2-\sum_i(n-b_i)^2\right]+1\notag
\end{align}
Consider now the pencil
\be\label{basic-pencil}
 |M|:=|3l-\sum_{i=1}^8 E_i|=|-K_S+\sum_{i\ge 9}E_i|
\ee
on  $S$. We use this pencil as a probe for  the Brill-Noether-Petri property.
Exactly as in the case of an elliptic ruled surface, we have the sequence (\ref{probe0}),
and the equalities:  $h^0(S,M-C)=0$ and, by generality of the eight points, $h^0(S,M)=2$. Thus
$M_C:=M\otimes\cO_C$ is at least a pencil.
Enforcing the Brill-Noether-Petri property on this pencil, means that we must assume that both
(\ref{probe1}) and 
(\ref{probe2}) hold.

 Moreover: 
\begin{align}
\deg(M_C)&= 
 (3\ell-\sum_{i\le 8} E_i)\cdot ((3n+\epsilon)\ell-\sum_1^h(n-b_i)E_i )\notag\\
&= n+3\epsilon + \sum_{i \le 8} b_i \notag
\end{align}

For $k \ge 0$ we let
\[
A_k = C+kK_S = (k-n)K_S+\epsilon\ell+\sum_i b_iE_i
\]

\begin{lemma}\label{L:ak}
  $A_k$ is effective for all $0\le k \le \nu_3=n-b_3$.
\end{lemma}

\begin{proof}
Since $-K_S$ is effective, it suffices to prove that $A_{\nu_3}$ is effective. We have:
\begin{align}
A_{\nu_3}&=-b_3K_S+\epsilon\ell+\sum_i b_iE_i \notag\\
&=(3b_3+\epsilon)\ell-(b_3-b_1)E_1-(b_3-b_2)E_2 + D\notag
\end{align}
where $D= \sum_{i\ge 4}(b_i-b_3)E_i$ is effective.  The inequality \eqref{E:bi2} can be rephrased as
\[
3b_3+\epsilon \ge (b_3-b_1)+(b_3-b_2)
\]
and this implies that   $(3b_3+\epsilon)\ell-(b_3-b_1)E_1-(b_3-b_2)E_2 $ is effective. 
\end{proof}

Item D) of Proposition  \ref{main-can-sect} follows immediately from   the following Corollary, which then concludes the proof of Proposition  \ref{main-can-sect}.

\begin{cor}\label{C:bi}
Suppose that $C$ is BNP and that $n \ge 2$.  Then only the following possibilities may occur:
\begin{itemize}
\item[(i)]   $\epsilon=0$, $\nu_1=\nu_2=\nu_3=n$,  and $n-1\le \nu_9 \le n$, $\,\,\qquad \qquad d=3n$
\item[(ii)]   $\epsilon=1$, and again $\nu_1=\nu_2=\nu_3=n$,  and $n-1\le \nu_9 \le n$, $d=3n+1$
\item[(iii)]  $\epsilon=1$, $\nu_1=n+1$, $\nu_2=\nu_3=n$,  and $n-1\le \nu_9 \le n$,$\,\quad\,\,\,\,d=3n+1$
\item[(iv)] $\epsilon=2$, $\nu_1=\nu_2=n+1$, $\nu_3=n$, and $n-1\le \nu_9 \le n$, $\,\quad\,\, d=3n+2$
\end{itemize}
Equivalently:
\begin{itemize}
\item[(i)] $\epsilon=0$, $b_1=b_2=b_3=0$ and  $0\le b_9 \le 1$.
\item[(ii)] $\epsilon=1$, $b_1=b_2=b_3=0$ and  $0\le b_9 \le 1$.
\item[(iii)] $\epsilon=1$, $b_1=-1$, $b_2=b_3=0$ and  $0\le b_9 \le 1$.
\item[(iv)] $\epsilon=2$, $b_1=b_2=-1$, $b_3=0$, and  $0\le b_9 \le 1$.
\end{itemize}
Furthermore $h \le 18$.
\end{cor}

\begin{proof}
Let $N = \ell-E_1$ and $N_C= N\otimes\cO_C$. Since $N_C$ is at least a pencil and $C$ is a BNP-curve, we must have $h^0(C,K_C-2N_C)=0$, which implies $h^0(S,C-2N)=0$.  We have:
\begin{align}
C-2N+\nu_3K_S&= A_{\nu_3}-2N \notag\\
&=(3b_3+\epsilon-2)\ell-(b_3-b_1-2)E_1-(b_3-b_2)E_2 + D \notag 
\end{align}
where $D= \sum_{i\ge 4}(b_i-b_3)E_i$ is effective. Since we have:
\[
C-2N= (3b_3+\epsilon-2)\ell-(b_3-b_1-2)E_1-(b_3-b_2)E_2 + D -\nu_3K_S
\]
and $-\nu_3K_S$ is effective, the divisor $F=(3b_3+\epsilon-2)\ell-(b_3-b_1-2)E_1-(b_3-b_2)E_2$ cannot be effective.  But we have, by \eqref{E:bi2}:
\be\label{ref}
(3b_3+\epsilon-2) \ge (b_3-b_1-2)+(b_3-b_2)
\ee
and therefore,  either $3b_3+\epsilon-2 < 0$, or $b_3-b_1-2<0$ (otherwise $F$ is effective, arguing as in the proof of Lemma \ref{L:ak}).  But since $b_3 \ge 0$ and $b_3-b_2 \ge 0$ 
the following cases may a priori occur. 

 $3b_3+\epsilon-2 < 0$  and then:

\begin{itemize}

\item   $\epsilon=0$, $b_1=b_2=b_3=0$.

\item $\epsilon=1$, $b_1=b_2=b_3=0$.
\item $\epsilon=1$, $b_1=-1$, $b_2=b_3=0$.

\end{itemize}

$3b_3+\epsilon-2 =0$,  $b_3-b_1-2<0$, and $b_3-b_2>0$; then:

\begin{itemize}

\item $\epsilon=2$, $b_1=b_2=-1$, $b_3=0$.

\end{itemize}

The   case $3b_3+\epsilon-2 >0$  cannot occur: in fact the other condition $b_3-b_1-2<0$ implies that $b_3-b_1 \le 1$, and therefore also $b_3-b_2 \le 1$, and any such choice implies that $F$ is effective.

\medskip

Now consider the pencil $|M_C|$ and impose the Petri condition to it. This means that  $h^0(C,K_C-2M_C)=0$, and 
since $-2M$ is ineffective,  this implies that  $h^0(S,C-2M)=0$.  On the other hand:
\begin{align}
C-2M +(\nu_3-2)K_S&=C+\nu_3K_S- 2\sum_{i\ge 9}E_i \notag\\
&= A_{\nu_3}- 2\sum_{i\ge 9}E_i \notag\\
\text{(since $b_3=0$)}\quad &=\epsilon\ell +\sum_{1\le i\le 8}b_iE_i+
\sum_{j\ge 9}(b_j-2)E_j\notag
\end{align}
Assume by contradiction that $b_9 \ge 2$. Then also $b_j \ge 2$ for all $j \ge 9$.  This implies  that $C-2M +(\nu_3-2)K_S$ is effective. But then 

\[
C-2M = [C-2M +(\nu_3-2)K_S] +(\nu_3-2)(-K_S)
\]
is effective as well since $\nu_3=n \ge 2$. Let us finally prove that, under our hypotheses, we have $h\leq18$.
A case-by-case inspection, using equation (\ref{bez}), shows that conditions   $(i)-(iv)$ imply $h\leq18$, with only one exception, namely when
$\epsilon=2$, $\nu_1=\nu_2=n+1$, $\nu_3=n$, $\nu_4=n-1$. In this  case $h=19$. 
In this case, however, from (\ref{{E:3d}}) and from the genus formula, we get $g=10n-6$, 
while, $\deg(M_C)= n+9$. Hence, imposing  that $|M_C|$ is a Brill-Noether pencil, gives $2n+18-10n+4\ge 0$, 
so that $n=2$. As a consequence, $C$ is a degree 8 plane curve of genus 14, having a triple point. 
Thus $C$ is not a Brill-Noether
curve and this case must be excluded.

 \end{proof}
 
\begin{acknowledgements}
We would like to thank Claire Voisin for her encouragement at a crucial stage of our project.
We also thank the referee for a very careful 
 reading and many useful remarks.
\end{acknowledgements}

 \bibliographystyle{amsalpha}
 \bibliography{BPVdV84}

\end{document}